\documentclass{amsproc}
\usepackage{amsmath,amsfonts}

\newtheorem{thm}{Theorem}[section]
\newtheorem{lem}[thm]{Lemma}

\newtheorem{definition}{Definition}[section]
\newtheorem{prop}[thm]{Proposition}
\newtheorem{rem}{Remark}[section]
\newtheorem{hyp}{Hypothesis}[section]

\usepackage{graphicx}
\usepackage{amsfonts,amssymb,amsmath,latexsym,epsfig,verbatim,pgf}
\usepackage[absolute]{textpos}
\usepackage{graphics,wrapfig,times}

\numberwithin{equation}{section}

\def\a{\alpha}
\def\b{\beta}
\def\d{\delta}
\def\e{\epsilon}
\def\g{\gamma}
\def\G{\Gamma}
\def\l{\lambda}
\def\o{\omega}
\def\O{\Omega}
\def\s{\sigma}
\def\t{\theta}

\def\Arg{{\text{Arg}}}

\def\C{{\mathbb C}}
\def\CC{{\mathcal C}}
\def\D{{\overline{\mathcal C}}}
\def\E{{\mathbb E}}
\def\F{{\mathcal F}}
\def\GT{{\text{GT}}}
\def\GUE{{\text{GUE}}}
\def\H{{\mathcal H}}

\def\Im{{\text{Im}}}

\def\MM{{\mathbb M}}
\def\N{{\mathbb N}} 
\def\P{{\mathbb P}}
\def\R{{\mathbb R}}
\def\Re{{\text{Re}}}
\def\S{{\mathfrak S}}
\def\supp{{\mbox{Supp}}}

\def\Z{{\mathbb Z}}

\def\an{{q_n}}

\newcommand{\p}[1]{\ldots,{\widehat{#1}},\ldots}

\begin{document}

\title{Universality properties of Gelfand-Tsetlin patterns}

\author{Anthony P. Metcalfe}
\address{Institutionen f\"{o}r Matematik, Royal Institute of Technology (KTH), 100 44
Stockholm, Sweden}
\email{metcalf@kth.se}

\begin{abstract}
A standard Gelfand-Tsetlin pattern of depth $n$ is a configuration of particles
in $\{1,\ldots,n\} \times \R$. For each $r \in \{1,\ldots,n\}$, $\{r\} \times \R$
is referred to as the $r^\text{th}$ level of the pattern. A standard Gelfand-Tsetlin
pattern has exactly $r$ particles on each level $r$, and
particles on adjacent levels satisfy an interlacing constraint.

Probability distributions on the set of Gelfand-Tsetlin patterns of depth $n$ arise
naturally as distributions of eigenvalue minor processes of random Hermitian matrices
of size $n$. We consider such probability spaces when the distribution of the matrix is
unitarily invariant, prove a determinantal structure for a broad subclass, and calculate
the correlation kernel.

In particular we consider the case where the eigenvalues of the random matrix are fixed.
This corresponds to choosing uniformly from the set of Gelfand-Tsetlin patterns whose
$n^\text{th}$ level is fixed at the eigenvalues of the matrix. Fixing $q_n \in \{1,\ldots,n\}$,
and letting $n \to \infty$ under the assumption that $\frac{q_n}n \to \a \in (0,1)$ and the
empirical distribution of the particles on the $n^\text{th}$ level converges weakly, the
asymptotic behaviour of particles on level $q_n$ is relevant to free probability theory.
Saddle point analysis is used to identify the set in which these particles behave
asymptotically like a determinantal random point field with the Sine kernel.
\end{abstract}

\maketitle

\section{Introduction}
\label{secI}

The spectrum of projections of random Hermitian matrices is an important object of study,
both in free probability and in random matrix theory. 
For each $n \in \N$, let $\H_n \subset \C^{n \times n}$ be the set of $n \times n$
Hermitian matrices, and let $A_n \in \H_n$ be a random matrix whose distribution is
unitarily invariant. For each $r \in \{1,\ldots,n\}$, let $\pi_r \in \C^{n \times n}$
be the diagonal projection of rank $r$ with the diagonal $(1,1,\ldots,1,0,0,\ldots,0)$.
Fix $q_n \in \{1,\ldots,n\}$, and let $n \to \infty$ under the assumption that
$\frac{q_n}n \to \a \in (0,1)$ and the empirical eigenvalue distribution of $A_n$
converges weakly to a compactly supported probability measure, $\mu$. The asymptotic behaviour
of the non-trivial eigenvalues of $\pi_{q_n} A_n \pi_{q_n}$ is of interest. In free probability,
the asymptotic behaviour can be used to study the free additive
convolution semi-group of $\mu$ (see Section \ref{secfp} for a brief introduction, and
Nica and Speicher, \cite{Nica06}, for a more comprehensive reference). In this paper
we identify the set in which the eigenvalues behave asymptotically like a
determinantal random point field with the Sine kernel.

The non-trivial eigenvalues of projections can be considered as particles in a random
interlaced system. For each $r \in \{1,\ldots,n\}$, let $\CC_r := \{ (y_1^{(r)}, \cdots,
y_r^{(r)}) \in \R^r : y_1^{(r)} > \cdots > y_r^{(r)} \}$, and $\l^{(r)} :=
(\l_1^{(r)},\cdots,\l_r^{(r)}) \in \D_r$ be the non-trivial eigenvalues of $\pi_r A_n \pi_r$.
Theorem 4.3.15 of Horn and Johnson, \cite{Horn90}, then gives
\begin{equation}
\label{eqSyInt}
\l_1^{(r+1)} \; \ge \; \l_1^{(r)} \; \ge \; \l_2^{(r+1)} \; \ge \; \l_2^{(r)}
\; \ge \cdots \ge \; \l_r^{(r)} \; \ge \; \l_{r+1}^{(r+1)},
\end{equation}
for all $r \in \{1,\ldots,n-1\}$. We write $\l^{(r+1)} \succeq \l^{(r)}$ for all $r$, and
say that the eigenvalues are {\em symmetrically interlaced}. Thus $(\l^{(1)},\ldots,\l^{(n)})
\in \overline{\GT_n}$ where
\begin{equation}
\label{eqGTnS}
\overline{\GT_n} := \left\{ (y^{(1)},\ldots,y^{(n)}) \in \D_1 \times \cdots \times \D_n :
y^{(n)} \succeq y^{(n-1)} \succeq \cdots \succeq y^{(1)} \right\}.
\end{equation}
This is referred to as the set of {\em standard Gelfand-Tsetlin patterns of depth $n$}. Figure
\ref{figGelfandTsetlin} gives an example of such a pattern.

\begin{figure*}
\centerline{\mbox{\includegraphics[width=2in]{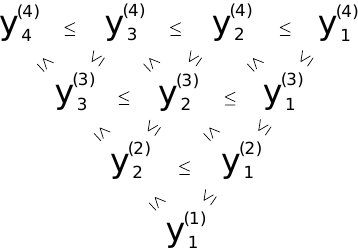}}}
\caption{A Gelfand-Tsetlin pattern, $(y^{(1)},y^{(2)},y^{(3)},y^{(4)})$, of depth $4$.}
\label{figGelfandTsetlin}
\end{figure*}

The interlaced $n$-tuple $(\l^{(1)},\ldots,\l^{(n)}) \in \overline{\GT_n}$ is referred
to as the {\em eigenvalue minor process} of $A_n$. Letting $\mu_n$ be the distribution
of $\l^{(n)} \in \D_n$ (i.e. the eigenvalue distribution of $A_n$), and 
assuming that $\mu_n$ is supported on $\CC_n$, it follows from
Baryshnikov, \cite{Bar01}, that $(\l^{(1)},\ldots,\l^{(n)})$ has distribution
\begin{equation}
\label{eqEigMinProBary}
d \nu_n [y^{(1)},\ldots,y^{(n)}]
= \begin{cases}
\prod_{i < j} \left( \frac{j - i}{y_i^{(n)} - y_j^{(n)}} \right)
d \mu_n [y^{(n)}] dy^{(n-1)} \ldots dy^{(1)} & ; y^{(n)} \in \CC_n, \\
0 & ; \mbox{otherwise},
\end{cases}
\end{equation}
for all $(y^{(1)},\ldots,y^{(n)}) \in \overline{\GT_n}$, where
$d y^{(r)}$ is Lebesgue measure on $\R^r$ for each $r$.
%
%
In the language of Baryshnikov, $\nu_n$ is the {\em uniform lift} of $\mu_n$ to
$\overline{\GT_n}$.

In Section \ref{secDSOGTP} we consider the case where $\nu_n$ can be written in the form
\begin{equation}
\label{eqGTnMeas}
d\nu_n[y^{(1)},\ldots,y^{(n)}] := \frac1{Z_n}
\det \left[ \phi_i (y_j^{(n)}) \right]_{i,j=1}^n dy^{(n)} dy^{(n-1)} \ldots dy^{(1)},
\end{equation}
for all $(y^{(1)},\ldots,y^{(n)}) \in \overline{\GT_n}$, where $\phi_1,\ldots,\phi_n
: \R \to \R$, and $Z_n > 0$ is a normalisation constant. Assuming integrability conditions
on $\phi_1,\ldots,\phi_n$, we prove that $(\overline{\GT_n},\nu_n)$ is a determinantal
random point field and calculate the correlation kernel (see Section \ref{secDRPF} for
an introduction to determinantal random point fields). Perhaps the best studied example of
such distributions is the eigenvalue minor process of the Gaussian unitary ensemble (GUE),
which we discuss in more detail in Section \ref{secTEOTMOTGUE}. In this case, as we shall
see, $\phi_i (y) = H_{n-i} (y) e^{- \frac12 y^2}$ for all $i \in \{1,\ldots,n\}$ and
$y \in \R$, where $H_i : \R \to \R$ is the Hermite polynomial of degree $i$.

Fixing $a,b \in \R$ with $a < b$, and $x^{(n)} \in \CC_n \cap [a,b]^n$ for all $n \in \N$,
consider the case where $\phi_i = \d_{x_i^{(n)}}$ for all $i \in \{1,\ldots,n\}$.
Then, the measure in equation (\ref{eqGTnMeas}) is the distribution of the eigenvalue
minor process of $U_n B_n U_n^\ast$, where $B_n \in \H_n$ is a fixed Hermitian matrix with
eigenvalues $x^{(n)}$, and $U_n \in \C^{n \times n}$ is a random Unitary matrix chosen
according to Haar measure. We are interested in the behaviour of $\l^{(q_n)}$ in the above
asymptotic limit (i.e. $n \to \infty$ under the assumption that $\frac{q_n}n \to \a \in (0,1)$
and the empirical distribution of $x^{(n)}$ converges weakly to $\mu$).

In Section \ref{secfp} we recall known results about the global asymptotic behaviour of
$\l^{(q_n)}$. It follows from the interlacing constraint that the empirical distribution
of $\l^{(q_n)}$ is supported on $[a,b]$. As we shall see, the expectation of the empirical
distribution converges weakly to a measure $\mu_\a$ on $[a,b]$ in the above asymptotic limit.
An expression for $\mu_\a$ in terms of the free additive convolution semi-group of $\mu$
follows from the work of Voiculescu, \cite{Voi98}, and a Lebesgue decomposition of $\mu_\a$
can be characterised from the work of Belinschi, \cite{Bel03}, \cite{Bel06}.

In this paper we consider the local asymptotic behaviour of $\l^{(q_n)}$. The main result
of this paper, described in detail in Section \ref{secR}, can be summarised as follows:
\begin{thm}
\label{thI}
For each $n \in \N$, let $K_n : \R^2 \to \C$ be the correlation kernel associated with $\l^{(q_n)}$.
Then for all $c \in (a,b)$ contained in that subset of the support of $\mu_\a$ on which $\mu_\a$
is absolutely continuous with respect to Lebesgue measure,
$$
\lim_{n \to \infty} \frac1{n \rho_{\a} (c)}
K_n \left( c + \frac{u}{n \rho_{\a} (c)}, c + \frac{v}{n \rho_{\a} (c)} \right)
= \frac{\sin(\pi (v - u))}{\pi (v - u)},
$$
where $\rho_\a(c)$ is the density of $\mu_\a$ at the point $c$.
\end{thm}
The limiting correlation kernel given above is referred to as the {\em Sine} kernel. This
has been observed asymptotically in the spectrum of other ensembles of random matrices and
in related systems (see, for example, \cite{Erdos2010a}, \cite{Joh01}, \cite{Pas97}).
Thus locally, as long as we avoid points where the non-trivial eigenvalues
accumulate (i.e. atoms of $\mu_\a$), the eigenvalues are asymptotically distributed as a
determinantal random point field with the Sine kernel. The strength of the above Theorem
is that the asymptotic behaviour can be observed without needing specific information
about $\mu$. This is a generalisation of Collins, \cite{Col05}, who took $B_n$ to be a
projection of rank $\tilde{q}_n$ with $\frac{\tilde{q}_n}n \to \b \in (0,1)$ as
$n \to \infty$. In this case $\mu = (1 - \b) \d_0 + \b \d_1$. This result is recovered
in Section \ref{secExTwoAtoms}.

Random systems with no obvious connection to random matrices sometimes give rise to related
measures. Examples include the bead model (see Boutillier, \cite{Bout06}), random tilings
(see, for example, \cite{Fer03}, \cite{Joh06a}, \cite{OkoRes06}) and
polynuclear growth (see Johansson, \cite{Joh03}). These models have subtle connections. For
example Johansson and Nordenstam, \cite{Joh06a}, \cite{Nord09}, consider random tilings of
a hexagon with lozenges. Lozenges are shown to interlace, and, in the large hexagon limit,
lozenges close to the boundary behave asymptotically like the eigenvalue minor process of
the GUE.

The paper is structured as follows: Sections \ref{secDRPF} and \ref{secTEOTMOTGUE} motivate
this topic by giving an introduction to determinantal random point fields, and by
discussing the GUE case in greater detail. Section \ref{secfp} recalls the known results
regarding the global behaviour of $\l^{(q_n)}$ in the above asymptotic limit. The
main result is stated in Section \ref{secR}. Section \ref{secEx} considers special cases of
the measure $\mu$.

Section \ref{secDSOGTP} contains the initial results on the determinantal structure
of the space $(\overline{\GT_n},\nu_n)$ when $\nu_n$ can be written in the form
given in equation (\ref{eqGTnMeas}). We also calculate the correlation kernel. Though the
main result of this section, Theorem \ref{thmGTDet}, follows from the more general results
of Defosseux, \cite{Man08}, we give a simplified account. We
obtain useful contour integral expressions for the correlation kernel in Proposition
\ref{prIntExpKl}.

Section \ref{secPOTMR} contains a proof of the main result, Theorem \ref{thSine}.
The asymptotic behaviour of the correlation kernel is obtained by performing a saddle point
analysis on the contour integral expression for the kernel given in Proposition
\ref{prIntExpKl}. Finally, in Section \ref{secUIE}
we consider the case where the measure on the Gelfand-Tsetlin patterns is induced by the
eigenvalue minor process of a Unitary invariant ensemble. In Section \ref{secTCE}
we specialise to classical ensembles that satisfy a Rodrigues formula. We recover
the correlation kernel of the eigenvalue minor process of the GUE obtained by Johansson and
Nordenstam, \cite{Joh06a} (see equation (\ref{eqJnGUE})).

\subsection{Determinantal random point fields}
\label{secDRPF}

The following is a brief introduction to determinantal random point fields. For a more
complete treatment see Johansson, \cite{Joh06b}, and Soshnikov, \cite{Sosh00}.

Let $E$ be a Polish space. Fix $N \in \N \cup \{\infty\}$, and let $\O \subset E^N$ be a space
of configurations of $N$-particles of $E$. The case $N = \infty$ gives countable configurations.
Denote each $\o \in \O$ by $(\o_1,\ldots,\o_N)$. We allow for multiple points, i.e.,
$\o_i = \o_j$ for $i \neq j$.

Given $\o \in \O$, and a Borel set $B \subset E$, define $N_B(\o) := \# \{i : \o_i \in B \}$,
the number of particles from $\o$ contained in $B$. We call $\o$ {\em locally finite} if
$N_K(\o)$ is finite for every compact set $K \subset E$. Assume $\O$ consists entirely of
locally finite configurations. Given $m \leq N$, define $C_B^m \subset \O$ by
$C_B^m := \{ \o \in \O : N_B(\o) = m \}$. This is called a \emph{cylinder set}. Let $\F$ be the
$\s$-algebra generated by the cylinder sets.

\begin{definition}
\label{defRanPoiFie}
A random point field is a triplet $(\O,\F,\P)$, where $\P$ is a probability measure on $(\O,\F)$.
\end{definition}

Let $(\O,\F,\P)$ be a random point field. For each $m \leq N$ define a measure,
$\MM_m$, on $E^m$ by
\begin{equation}
\label{eqMMm}
\MM_m[B] := \E \left[ \sum_{1 \le i_1 \ne \cdots \ne i_m \le N}
1_{\{\o \in \O : (\o_{i_1},\ldots,\o_{i_m}) \in B\}} \right],
\end{equation}
for any Borel subset $B \subset E^m$. We assume that $\MM_m$ is well-defined for all $m$,
and $\MM_m[B] < \infty$ whenever $B$ is bounded. For each $m \leq N$, and each Borel subset
$B \subset E^m$, $\MM_m[B]$ is the expected number of $m$-tuples of particles from $\O$
that are contained in $B$. Also, for all $m \leq N$, and all disjoint bounded Borel sets
$B_1,\ldots,B_m \subset E$,
$$
\MM_m [B_1 \times \cdots \times B_m] = \E \left[ \prod_{k=1}^m N_{B_k} \right].
$$

Letting $\mu$ be a reference measure on $E$, for example Lebesgue on $\R$, we make the following definition:
\begin{definition}
\label{defCorFnt}
For any $m \le N$, the Radon-Nikodym derivative of $\MM_m$ with respect to $\mu^m$ (if it exists)
is referred to as the $m^\text{th}$ correlation function of the random point field. That is, the
$m^\text{th}$ correlation function is the integrable function $\rho_m : E^m \rightarrow \R$ which
satisfies
$$
\MM_m [B] = \int_B \rho_m(y_1,\ldots,y_m) d\mu^m[y],
$$
for all Borel subsets $B \subset E^m$.
\end{definition}
This property is useful, for example, when calculating \emph{last particle distributions}.
That is, the distribution of the rightmost particle of random point fields over $\R$.
See Johansson, \cite{Joh06b}, for more details.

\begin{definition}
\label{defDetPointPro}
A random point field is called determinantal if all correlation functions
exist and there exists a function $K : E^2 \rightarrow \C$ for which
$$
\rho_m(y_1,\ldots,y_m) = \det[K(y_i,y_j)]_{i,j=1}^m,
$$
for all $y_1,\ldots,y_m \in E$ and $m \leq N$. $K$ is called the correlation kernel of the field.
\end{definition}

\begin{rem}
When $\F$ and $\mu$ are `obvious' they are not usually mentioned.
For example when $E \subset \R$, $\F$
is the Borel sigma-algebra and $\mu$ is Lebesgue measure. When
$E \subset \Z \times \R$, $\F = \{ A \times B : A \subset \Z \mbox{ and }
B \subset \R \mbox{ is Borel } \}$ and $\mu$ is the direct product of the counting
measure and Lebesgue measure.
\end{rem}

\begin{rem}
\label{remCorKer}
Correlation kernels are not necessarily unique. For example when $E \subset \R$, another correlation
kernel $J : E^2 \rightarrow \C$ can be defined
by $J(u,v) := \frac{w(u)}{w(v)} K(u,v)$ for all $u,v \in \R$, where $w$ is any non-zero
complex function. 
\end{rem}

\subsection{The eigenvalue minor process of the GUE}
\label{secTEOTMOTGUE}

The GUE is the probability measure on $\H_n$ given by
$$
d\xi_n^\GUE[H] := \frac1{Z_n} e^{-\frac12 \text{Tr} H^2} dH,
$$
where $Z_n>0$ is a normalisation constant, and $dH$ is the Lebesgue measure
$$
dH := \left( \prod_{i=1}^n d(H_{ii}) \right) \left( \prod_{j<k} d(\Re H_{jk}) d(\Im H_{jk}) \right).
$$
A typical matrix chosen according to the GUE has diagonal elements given by independent standard
Gaussians, and the real and imaginary part of the non-diagonal elements given by independent
Gaussians with variance $\frac12$.

Let $(\l^{(1)},\ldots,\l^{(n)}) \in \overline{\GT_n}$ be the eigenvalue minor process of
the GUE, as discussed in Section \ref{secI}. The distribution of
$\l^{(n)} \in \D_n$ (i.e. the distribution of the eigenvalues of the GUE) is
given by (see for example Mehta, \cite{Mehta04})
\begin{equation}
\label{eqfnGUE}
d \mu_n^\GUE [y] = \frac1{Z_n'} \Delta_n(y)^2 \left( \prod_{i=1}^n e^{-\frac12 y_i^2} \right) dy,
\end{equation}
for all $y \in \D_n$, where $Z_n' > 0$ is a normalisation constant, $dy$ is Lebesgue measure
on $\R^n$, and $\Delta_n : \R^n \to \R$ is the {\em Vandermonde} determinant
\begin{equation}
\label{eqVanDet}
\Delta_n(y)
\; := \; \prod_{1 \le i < j \le n} (y_i - y_j)
\; = \; \det \left[ y_i^{n-j} \right]_{i,j=1}^n.
\end{equation}
Equation (\ref{eqEigMinProBary}) thus implies that $(\l^{(1)},\ldots,\l^{(n)})$
has distribution
$$
d \nu_n^\GUE [y^{(1)},\ldots,y^{(n)}] = \frac1{Z_n''} \Delta_n(y^{(n)})
\left( \prod_{i=1}^n e^{-\frac12  (y_i^{(n)})^2} \right) \; dy^{(n)} dy^{(n-1)} \ldots dy^{(1)},
$$
for all $(y^{(1)},\ldots,y^{(n)}) \in \overline{\GT_n}$, where $Z_n'' > 0$ is
a normalisation constant, and $dy^{(r)}$ is Lebesgue measure on $\R^r$ for each $r$.

Definition \ref{defRanPoiFie} implies that $(\D_n,\mu_n^\GUE)$ is a random
point field on $\R$. Let $\{H_i\}_{i\ge 0}$ be the sequence of monic Hermite polynomials,
i.e., for each $i,j \ge 0$, $H_i$ and $H_j$ have degree $i$ and $j$ respectively and satisfy
$$
\int_{-\infty}^\infty H_i(y) H_j(y) e^{-\frac12 y^2} dy = \sqrt{2\pi i! j!} \; \delta_{ij}.
$$
Equations (\ref{eqfnGUE}) and (\ref{eqVanDet}) then give
$$
d\mu_n^\GUE[y] = \frac1{Z_n'}
\left( \det \left[ H_{n-j} (y_i) e^{-\frac14 y_i^2} \right]_{i,j=1}^n \right)^2 dy,
$$
for all $y \in \D_n$.
%
%
Proposition 2.11 of Johansson, \cite{Joh06b}, then shows that this field is determinantal
with correlation kernel $K_n^\GUE : \R^2 \to \R$ given by
\begin{equation}
\label{eqKnGUE}
K_n^\GUE(u,v)
= \sum_{i=0}^{n-1} \frac1{\sqrt{2\pi} \; i!} H_i (u) H_i (v) e^{-\frac14 (u^2 + v^2)},
\end{equation}
for all $u,v \in \R$.
%
%

More recently Johansson and Nordenstam, \cite{Joh06a}, showed a determinantal structure
for $(\overline{\GT_n},\nu_n^\GUE)$.
For simplicity of notation identify $\overline{\GT_n}$ with a space of
configurations of $\frac12 n(n+1)$ particles on $\{1,\ldots,n\} \times \R$ using the natural map from
$\overline{\GT_n}$ to $(\{1,\ldots,n\} \times \R)^{\frac12 n(n+1)}$ given by
$$
(y^{(1)},\ldots,y^{(n)}) \mapsto
\left( (1,y_1^{(1)}),(2,y_1^{(2)}),(2,y_2^{(2)}),
(3,y_1^{(3)}),(3,y_2^{(3)}),(3,y_3^{(3)}),\ldots \right),
$$
for all $(y^{(1)},\ldots,y^{(n)}) \in \overline{\GT_n}$. In words, the first
particle of each configuration is contained in $\{1\} \times \R$, the next $2$ particles are
contained in $\{2\} \times \R$, next $3$ in $\{3\} \times \R$ etc.
Definition \ref{defRanPoiFie} thus implies that $(\overline{\GT_n},\nu_n^\GUE)$
is a random point field on $\{1,\ldots,n\} \times \R$. Johansson and Nordenstam, \cite{Joh06a},
show that this field is determinantal with correlation kernel
$J_n^\GUE : (\{1,\ldots,n\} \times \R)^2 \to \R$ given by
\begin{eqnarray}
\label{eqJnGUE}
\lefteqn{\hspace{1.5cm} J_n^\GUE((r,u),(s,v))
\; = \; \sum_{i=-\infty}^{-1} \frac1{\sqrt{2\pi}(i+s)!} H_{i+r}(u) H_{i+s}(v) e^{-\frac14 (u^2+v^2)}} \\
\nonumber
& + & 1_{s>r} e^{\frac14 (u^2-v^2)} \left( \sum_{i=-s}^{-r-1} \frac{H_{i+s} (v)}{\sqrt{2\pi}(i+s)!}
\int_u^\infty dx \frac{(x-u)^{-i-r-1}}{(-i-r-1)!} e^{- \frac12 x^2}
- \frac{(v-u)^{s-r-1}}{(s-r-1)!} 1_{v>u} \right),
\end{eqnarray}
for all $r,s \in \{1,\ldots,n\}$ and $u,v \in \R$. Similar correlation kernels have been
obtained for the eigenvalue minor processes of Jacobi and Laguerre ensembles (see equation
(4.15) of Forrester and Nagao, \cite{Forr11}). Section \ref{secTCE} provides an
alternative method for calculating these kernels.

As a final note we would like to point out some interesting asymptotics that are of relevance
to our problem. For more information, see Anderson, Guionnet and Zeitouni, \cite{And10}:
\begin{thm}
\label{thmGUEEmpEigAsy}
Let $\mu_\text{sc}$ be the semicircle distribution, i.e., the distribution on $\R$
with density $\rho_\text{sc} : \R \to \R$ given by 
\begin{equation}
\label{eqSemiCir}
\rho_\text{sc} (c) := \frac1{2\pi} \sqrt{4-c^2} 1_{|c| \le 2},
\end{equation}
for all $c \in \R$. Then as $n \to \infty$,
$$
\frac1n \sum_{j=1}^n \delta_{\l_j^{(n)}/\sqrt{n}} \to \mu_\text{sc}
\hspace{0.5cm} \mbox{almost surely,}
$$
in the sense of weak convergence of measures.
\end{thm}

\begin{thm}
\label{thmGUECorKerAsy}
For any $c \in (-2,2)$, and any sequence $\{c_n\}_{n\ge 1} \subset \R$ with
$\frac{c_n}{\sqrt{n}} \to c$,
$$
\lim_{n \to \infty} \frac1{\rho_\text{sc} (c) \sqrt{n}}
K_n^{\GUE} \left( c_n + \frac{u}{\rho_\text{sc} (c) \sqrt{n}},
c_n + \frac{v}{\rho_\text{sc} (c) \sqrt{n}} \right) = \frac{\sin(\pi(v-u))}{\pi(v-u)},
$$
for all $u,v \in \R$.
\end{thm}

For each $r \in \{1,\ldots,n\}$, $J_n^\GUE((r,\cdot),(r,\cdot)) : \R^2 \to \R$
is the correlation kernel for the particles on level $r$ of the interlaced pattern (i.e. the
eigenvalues of the sub-matrix of size $r$). Equations (\ref{eqKnGUE})
and (\ref{eqJnGUE}) give $J_n^\GUE((r,\cdot),(r,\cdot)) = K_r^\GUE$, and so
the particles on level $r$ are distributed as the eigenvalues of a randomly chosen
GUE matrix of size $r$. Therefore, properly rescaled, the particles in the bulk on each level
of the interlaced pattern behave asymptotically like a determinantal random point
field with the Sine kernel.

Related systems of interlaced particles often display similar asymptotic behaviour.
For example Boutillier, \cite{Bout06}, studies the {\em bead model}, a probability measure
on systems of interlaced particles on $\Z \times \R$. The particles on each thread (i.e.
on $\{r\} \times \R$ for each $r$) form a determinantal random point field with the Sine
kernel. In Metcalfe, O'Connell and Warren, \cite{Met09}, a circular analogue of this model
is constructed.

\subsection{Global asymptotic behaviour of the eigenvalues of random projections}
\label{secfp}

Fix $a,b \in \R$ with $a<b$. For each
$n \in \N$, fix $q_n \in \{1,\ldots,n\}$, $x^{(n)} \in \CC_n \cap [a,b]^n$ and
$B_n \in \H_n$ with eigenvalues $x^{(n)}$. Let $U_n \in \C^{n \times n}$ be a random
Unitary matrix chosen according to Haar measure, and let $(\l^{(1)},\ldots,\l^{(n)})
\in \overline{\GT_n}$ be the eigenvalue minor process of $U_n B_n U_n^\ast$, as
discussed in Section \ref{secI}. In this section we recall known results about the
behaviour of the empirical distribution of $\l^{(q_n)}$ under the following asymptotic
limit:
\begin{hyp}
\label{hypWeakConv}
Let $\mu$ be a probability measure on $\R$ which is not a point mass and with support,
$\supp (\mu) \subset [a,b]$. Assume that, as $n \to \infty$,
$$
\frac1n \sum_{i=1}^n \delta_{x_i^{(n)}}
\rightarrow \mu,
$$
in the sense of weak convergence of measures. Also assume that there exists an $\a \in (0,1)$
for which $\frac{q_n}n \to \a$ as $n \to \infty$.
\end{hyp}

It follows from the interlacing constraint (see equation (\ref{eqSyInt})) that
$\l^{(q_n)} \in [a,b]^{q_n}$. Let $\MM_1^{(n)}$ be the measure on $[a,b]$ of size $q_n$ for which,
for any $B \subset [a,b]$ measurable, $\MM_1^{(n)}[B]$ is the expected number of eigenvalues from
$\{\l_1^{(q_n)}, \ldots, \l_{q_n}^{(q_n)} \}$ that are contained in $B$ (see equation
(\ref{eqMMm})). The following is a consequence of Voiculescu, \cite{Voi98}.
For more information see Xu, \cite{Xu97}, and Collins, \cite{Col03a}, \cite{Col03b}, \cite{Col05}:
\begin{lem}
\label{lemVoiWeakCon}
Assuming hypothesis \ref{hypWeakConv}, $\frac1n \MM_1^{(n)}$ converges weakly to $\mu_\a$,
the measure on $[a,b]$ of size $\a$ given by
$$
(1 - \a) \d_0 + \mu_\a = \left( (1 - \a) \d_0 + \a \d_1 \right) \boxtimes \mu,
$$
where $\boxtimes$ represents free multiplicative convolution.
\end{lem}

For more information on free multiplicative convolution see Nica and Speicher,
\cite{Nica06}, lecture $14$. Exercise $14.21$ of this book gives an alternative
expression for $\mu_\a$:
\begin{equation}
\label{eqMulCon0}
\mu_\a = \a \; D_\a ( \mu^{\boxplus \a^{-1}} ),
\end{equation}
where $D_\a$ is the dilation operator that satisfies $D_\a \delta_a := \delta_{\a a}$
for all $a \in \R$, $\boxplus$ represents free additive convolution, and
$\{ \mu^{\boxplus t} \}_{t \ge 1}$ is the free additive convolution semi-group of $\mu$
(i.e. $\mu^{\boxplus 1} = \mu$, $\mu^{\boxplus (s + t)}
= (\mu^{\boxplus s}) \boxplus (\mu^{\boxplus t})$ for all $s,t \ge 1$, and the
mapping $t \mapsto \mu^{\boxplus t}$ is continuous with respect to the $\mbox{weak}^\ast$
topology on probability measures).

A Lebesgue decomposition of $\mu_\a$ follows from Belinschi, \cite{Bel03} (Theorem $4.1$),
\cite{Bel06} (Theorem $1.36$):
\begin{lem}
\label{lemLebDec}
$\mu_\a = \mu_\a^\text{at} + \mu_\a^\text{ac} + \mu_\a^\text{sc}$ where
\begin{enumerate}
\item
$\mu_\a^\text{at}$ is an atomic measure with support $\supp (\mu_\a^\text{at})
= \{c \in [a,b] : \mu[\{c\}] > 1 - \a\}$. Moreover $\mu_\a^\text{at} [\{c\}]
= \mu[\{c\}] - (1 - \a)$ for all $c \in \supp (\mu_\a^\text{at})$.
\item
$\mu_\a^\text{ac}$ is a non-zero measure which is absolutely continuous with
respect to Lebesgue measure, and its density is analytic outside a closed set of
Lebesgue measure zero.
\item
$\mu_\a^\text{sc}$ is singular continuous with respect to Lebesgue measure. Moreover the
support of $\mu_\a^\text{sc}$ has zero Lebesgue measure, and is included in
the support of $\mu_\a^\text{ac}$.
\end{enumerate}
\end{lem}

\subsection{Statement of the main result}
\label{secR}

Fix $a,b \in \R$ with $a<b$. For each
$n \in \N$, fix $q_n \in \{1,\ldots,n\}$, $x^{(n)} \in \CC_n \cap [a,b]^n$ and
$B_n \in \H_n$ with eigenvalues $x^{(n)}$. Let $U_n \in \C^{n \times n}$ be a random
Unitary matrix chosen according to Haar measure, and let $(\l^{(1)},\ldots,\l^{(n)})
\in \overline{\GT_n}$ be the eigenvalue minor process of $U_n B_n U_n^\ast$. Assume hypothesis
\ref{hypWeakConv}. In this section we consider the local asymptotic behaviour of $\l^{(q_n)}$
as $n \to \infty$.

Equation (\ref{eqEigMinProBary}) implies that $(\l^{(1)},\ldots,\l^{(n)})$ has distribution
\begin{equation}
\label{eqUniMeasGTn}
d\nu_n[y^{(1)},\ldots,y^{(n)}] := \frac1{Z_n}
\delta_{x^{(n)}} (y^{(n)}) dy^{(n)} dy^{(n-1)} \ldots dy^{(1)},
\end{equation}
for all $(y^{(1)},\ldots,y^{(n)}) \in \overline{\GT_n}$, where $Z_n>0$ is a normalisation
constant and $dy^{(r)}$ is Lebesgue measure on $\R^r$ for each $r$. As in the GUE
case (see Section \ref{secTEOTMOTGUE}), we identify $\overline{\GT_n}$ with a space of
configurations of $\frac12 n(n+1)$ particles on $\{1,\ldots,n\} \times [a,b]$. Note, we
restrict our attention to $[a,b]$ since $x^{(n)} \in [a,b]^n$, and so
$\l^{(r)} \in [a,b]^r$ (for each $r$) by the interlacing constraint (see equation (\ref{eqSyInt})).
Definition \ref{defRanPoiFie} implies that $(\overline{\GT_n},\nu_n)$ is a random point field
on $\{1,\ldots,n\} \times [a,b]$. Theorem \ref{thmGTDet} and remark \ref{remphiEx}
show that this field is determinantal with correlation kernel
$K_n : (\{1,\ldots,n\} \times [a,b])^2 \to \C$ given by
\begin{eqnarray}
\label{eqKnrusvFixTopLine}
K_n ((r,u),(s,v))
& = & \sum_{j=1}^n 1_{v \le u < x_j^{(n)}} \frac{(x_j^{(n)} - u)^{n-r-1}}{(n-r-1)!}
\frac{\partial^{n-s}}{\partial v^{n-s}}
\prod_{i \neq j} \left( \frac{v - x_i^{(n)}}{x_j^{(n)} - x_i^{(n)}} \right) \\
\nonumber
& - & \sum_{j=1}^n 1_{v > u > x_j^{(n)}} \frac{(x_j^{(n)} - u)^{n-r-1}}{(n-r-1)!}
\frac{\partial^{n-s}}{\partial v^{n-s}}
\prod_{i \neq j} \left( \frac{v - x_i^{(n)}}{x_j^{(n)} - x_i^{(n)}} \right),
\end{eqnarray}
for all $r \in \{1,\ldots,n-1\}$, $s \in \{1,\ldots,n\}$ and $u,v \in [a,b]$.

For each $n \in \N$, $K_n((q_n,\cdot),(q_n,\cdot)) : [a,b]^2 \to \C$ is the
correlation kernel for $\l^{(q_n)}$, or equivalently the particles on level $q_n$ of the
Gelfand-Tsetlin pattern chosen according to the measure $\nu_n$. We wish to establish
a natural subset of $(a,b)$ under which this kernel behaves asymptotically like
the Sine kernel as $n \to \infty$. We define
\begin{equation}
\label{eqAa}
A_\a := \{ c \in (a,b) : \exists \; w \in \C \setminus \R \mbox{ with } w G_\mu (w + c) = 1 - \a \},
\end{equation}
where $G_\mu : \C \setminus \R \to \C$ is the Cauchy transform of $\mu$ (also known
as the Stieltjes transform) given by
\begin{equation}
\label{eqCauTra}
G_\mu (w) := \int_{-\infty}^\infty \frac1{w - x} \mu[dx],
\end{equation}
for all $w \in \C \setminus \R$. Proposition \ref{prAa} gives a natural interpretation of $A_\a$.

The main result (shown in section \ref{secPOTMR}) can now be stated as follows:
\begin{thm}
\label{thSine}
Assume hypothesis \ref{hypWeakConv}. Then, given $c \in A_\a$, there exists a $w_{\a,c} \in \C$
with $\Im(w_{\a,c}) > 0$ and
$$
\{ w \in \C \setminus \R : w G_\mu (w + c) = 1 - \a \} = \{ w_{\a,c}, \overline{w_{\a,c}} \}.
$$
Moreover for all $c \in A_\a$, and compact sets $U,V \subset \R$,
$$
\lim_{n \to \infty} \sup_{u \in U, v \in V} \left| \frac{(C_{\a,c})^{v-u}}{n \rho_{\a} (c)}
K_n \left( \left( q_n, c + \frac{u}{n \rho_{\a} (c)} \right),
\left( q_n, c + \frac{v}{n \rho_{\a} (c)} \right) \right)
- \frac{\sin(\pi (v - u))}{\pi (v - u)} \right| = 0,
$$
where $\rho_{\a} (c) := - \frac{1-\a}{\pi} \Im \left( w_{\a,c}^{-1} \right)$ and
$C_{\a,c} := \exp \left( \pi \; \frac{ \Re (w_{\a,c}^{-1}) }{ \Im (w_{\a,c}^{-1}) } \right)$.
\end{thm}

Natural interpretations exist for $A_\a$ and $\rho_\a : A_\a \to (0,\infty)$. Let $\mu_\a$
be the measure on $[a,b]$ of size $\a$ given in Lemma \ref{lemVoiWeakCon},
and let $\mu_\a^\text{at}$ be its atomic part (see Lemma \ref{lemLebDec}).
Then, letting $\supp$ represent support and ${}^\circ$ represent interior:
\begin{prop}
\label{prAa}
Assume hypothesis \ref{hypWeakConv}. Then $A_\a$ is open,
$A_\a \cap \supp (\mu_\a^\text{at}) = \emptyset$, $A_\a \subset \supp(\mu_\a)^\circ$,
and $\supp(\mu_\a)^\circ \setminus A_\a$ has Lebesgue measure zero. Moreover there exists
an open subset of $A_\a$, of equal Lebesgue measure, in which $\mu_\a$ is absolutely continuous
with respect to Lebesgue measure, and $\rho_\a (c)$ is the density of $\mu_\a$ at $c$ for each
$c$ in this set.
\end{prop}

To show this we consider the Cauchy transform and the $\mathcal{R}$-transform of
$\mu_\a$. Letting $\nu$ be a probability measure on $\R$ with compact support, the
$\mathcal{R}$-transform of $\nu$ is the function, $\mathcal{R}_\nu : \C \to \C$,
given by
$$
\mathcal{R}_\nu (w) := \sum_{n \ge 0} \kappa_{n+1} w^n
$$
for all $w \in \C$, where $\{\kappa_n\}_{n \ge 1}$ are the free cumulants of $\nu$
(see Nica and Speicher, \cite{Nica06}, lecture 12, for more information). The following
properties will be of use:
\begin{lem}
\label{lemStielTra}
For any two probability measure $\nu, \xi$ on $\R$ with compact support, and any $s \ge 1$,
we have $\mathcal{R}_{\nu^{\boxplus s}} = s \mathcal{R}_\nu$ and $\mathcal{R}_{\nu \boxplus \xi}
= \mathcal{R}_\nu + \mathcal{R}_\xi$. Moreover, $G_\nu : \C \setminus \R \to \C \setminus \R$
is invertible with inverse
$$
G_\nu^{-1} (w) = \mathcal{R}_\nu(w) + \frac1w,
$$
for all $w \in \C \setminus \R$.
\end{lem}

Lemma \ref{lemStielTra} gives
$$
w
= G_\mu \left( \a (G_{\mu^{\boxplus \a^{-1}}})^{-1} (w) + \frac{1-\a}w \right),
$$
for all $w \in \C \setminus \R$. 
%
%
Then, replacing $w$ by $G_{\mu_\a} (w)$, and noting that
$G_{\mu_\a} (w) = G_{\mu^{\boxplus \a^{-1}}} \left( \frac{w}\a \right)$ (see equation (\ref{eqMulCon0})),
%
%
$$
G_{\mu_\a} (w) = G_\mu \left( w + \frac{1 - \a}{G_{\mu_\a} (w)} \right),
$$
for all $w \in \C \setminus \R$.
%
%

Lemma \ref{lemLebDec} implies that there exists an open subset of $\supp(\mu_\a)^\circ$, of
equal Lebesgue measure, in which $\mu_\a$ is absolutely continuous with respect to Lebesgue
measure and the density of $\mu_\a$ is continuous.
%
%
We extend the Cauchy transform, $G_{\mu_\a} : \C \setminus \R \to \C$, to this set by defining
\begin{equation}
\label{eqCauTraExt}
G_{\mu_\a} (c) := \lim_{\e \to 0_+} G_{\mu_\a} ( c - i \e ),
\end{equation}
for all $c$ in the set. This is well-defined with $\frac1\pi \Im(G_{\mu_\a} (c))$ equal to
the density of $\mu_\a$ at $c$. Therefore
\begin{equation}
\label{eqGmuaGmu}
G_{\mu_\a} (c) = G_\mu \left( c + \frac{1 - \a}{G_{\mu_\a} (c)} \right),
\end{equation}
for all such $c$.
%
%

\begin{proof}[Proof of Proposition \ref{prAa}:]
For each $c \in (a,b)$, define $f_{\a,c} : \C \setminus \R \to \C$ by
\begin{equation}
\label{eqfac}
f_{\a,c} (w) = G_\mu (w + c) - \frac{1 - \a}w,
\end{equation}
for all $w \in \C \setminus \R$. Then $f_{\a,c}$ is analytic and $c \in A_\a$ if and only if
roots of $f_{\a,c}$ exist (see equation (\ref{eqAa})). Moreover, given $c \in A_\a$, there
exists a $w_{\a,c} \in \C$ with $\Im(w_{\a,c}) > 0$ and
$\{w \in \C \setminus \R : f_{\a,c}(w) = 0 \} = \{w_{\a,c}, \overline{w_{\a,c}} \}$
(see Theorem \ref{thSine}).

Fix $c \in A_\a$. Since $f_{\a,c}$ is a non-constant analytic function in $\C \setminus \R$
with $f_{\a,c}(w_{\a,c}) = 0$, there exists an $\e \in (0,\Im(w_{\a,c}))$ with
$f_{\a,c}(w) \neq 0$ for all $w \in \bar{B}(w_{\a,c},\e) \setminus \{w_{\a,c}\}$. Thus, letting
$\partial B(w_{\a,c},\e)$ be the boundary of $B(w_{\a,c},\e)$, the Bolzano-Weierstrass
Theorem gives
$$
\inf_{w \in  \partial B(w_{\a,c},\e)} |f_{\a,c}(w)| > 0.
$$
%
%
Rouch\'{e}'s Theorem (see Rudin, \cite{Rudin87}) and equation (\ref{eqfac}) thus imply
that there exists a $\d > 0$ for which $f_{\a,c}$ and $f_{\a,y}$ have the same number of
roots in $B(w_{\a,c},\e)$ for all $y \in (c-\d,c+\d)$.
%
%
Therefore $(c-\d,c+\d) \subset A_\a$, and so $A_\a$ is open.
Also equations (\ref{eqCauTra}) and (\ref{eqfac}) give
$$
1 - \a = \int \frac{w_{\a,c}}{w_{\a,c} + c - x} d\mu[x].
$$
Comparing real and imaginary parts gives
$$
1 - \a
= \int \frac{|w_{\a,c}|^2}{|w_{\a,c} + c - x|^2} d\mu[x]
= \mu[\{c\}] + \int_{[a,b] \setminus \{c\}} \frac{|w_{\a,c}|^2}{|w_{\a,c} + c - x|^2} d\mu[x].
$$
%
%
Thus, since $\mu$ is not a point mass (see hypothesis \ref{hypWeakConv}),
$\mu[\{c\}] < 1 - \a$. Lemma \ref{lemLebDec} thus gives $c \not\in \supp(\mu_\a^\text{at})$,
and so $A_\a \cap \supp(\mu_\a^\text{at}) = \emptyset$.

We now show that $A_\a \subset \supp(\mu_\a)^\circ$. Fix $c \in A_\a$.
Then, since $A_\a$ is open and $A_\a \cap \supp(\mu_\a^\text{at}) = \emptyset$, $[c-\d,c+\d]$ is a
continuity set of $\mu_\a$ for all $\d>0$ sufficiently small. Therefore Lemma \ref{lemVoiWeakCon}
implies that
$$
\mu_\a [c-\d,c+\d] = \lim_{n \to \infty} \frac1n \MM_1^{(n)} [c-\d,c+\d],
$$
for all $\d > 0$ sufficiently small, where $\MM_1^{(n)}$ is the measure on $[a,b]$ of size $q_n$
for which, for any $B \subset [a,b]$ measurable, $\MM_1^{(n)}[B]$ is the expected number of
eigenvalues from $\{\l_1^{(q_n)}, \ldots, \l_{q_n}^{(q_n)} \}$ that are contained in $B$
(see equation (\ref{eqMMm})). Then, since $K_n((q_n,\cdot),(q_n,\cdot)) : [a,b]^2 \to \C$ is the
correlation kernel for $\l^{(q_n)}$, definitions \ref{defCorFnt} and \ref{defDetPointPro} give
$$
\mu_\a [c-\d,c+\d] = \lim_{n \to \infty} \int_{c-\d}^{c+\d} \frac1n K_n ((q_n,y),(q_n,y)) dy,
$$
for all $\d > 0$ sufficiently small. Also, a slight extension of Theorem \ref{thSine} (shown in
the same way) gives
$$
\lim_{n \to \infty} \sup_{y \in [c-\d,c+\d]}
\left| \frac1{n \rho_{\a} (y)} K_n ((q_n,y),(q_n,y)) - 1 \right| = 0,
$$
for all $\d > 0$ sufficiently small, where $\rho_{\a} (y) =
- \frac{1-\a}{\pi} \Im \left( w_{\a,y}^{-1} \right)$, and so
$$
\mu_\a [c-\d,c+\d] = \int_{c-\d}^{c+\d} \rho_\a(y) dy.
$$
%
%
Thus, since $\rho_\a(y) > 0$ for all $y$, $\mu_\a [c-\d,c+\d] > 0$ for all $\d > 0$
sufficiently small, and so $c \in \supp (\mu_\a)$. This is true for all $c \in A_\a$,
and $A_\a$ is open, and so $A_\a \subset \supp(\mu_\a)^\circ$.

We now show that $\supp(\mu_\a)^\circ \setminus A_\a$ has Lebesgue measure zero.
Lemma \ref{lemLebDec} implies that there exists an open subset of $\supp(\mu_\a)^\circ$, of
equal Lebesgue measure, in which $\mu_\a$ is absolutely continuous with respect to Lebesgue
measure and the density of $\mu_\a$ is continuous. For all $c$ in this set, $G_{\mu_\a} (c)$
is well-defined and $\frac1\pi \Im(G_{\mu_\a} (c))$ is the density of $\mu_\a$ at $c$ (see
equation (\ref{eqCauTraExt})). For all such $c$, equations (\ref{eqGmuaGmu}) and (\ref{eqfac})
show that $f_{\a,c}$ has a root in $\C \setminus \R$ given by
$$
\frac{1-\a}{G_{\mu_\a} (c)}.
$$
Thus all such $c$ are in $A_\a$, and so $\supp(\mu_\a)^\circ \setminus A_\a$ has Lebesgue measure
zero. It remains to show that $\rho_\a(c)$ equals $\frac1\pi \Im(G_{\mu_\a} (c))$ for all such
$c$, the density of $\mu_\a$ at $c$. This follows by noting that
$\overline{w_{\a,c}} = \frac{1-\a}{G_{\mu_\a} (c)}$ (recall that there exists a
$w_{\a,c} \in \C$ with $\Im(w_{\a,c}) > 0$ and
$\{w \in \C \setminus \R : f_{\a,c}(w) = 0 \} = \{w_{\a,c}, \overline{w_{\a,c}} \}$)
and $\rho_{\a} (c) = - \frac{1-\a}{\pi} \Im \left( w_{\a,c}^{-1} \right)$.
%
%
\end{proof}

\subsection{Examples}
\label{secEx}

In this section we examine Theorem \ref{thSine} in some special cases:

\subsubsection{Semicircle distribution}

Fix $\a \in (0,1)$ and let $\mu$ be the semicircle distribution given in equation (\ref{eqSemiCir}).
Using the well known formula for the Cauchy transform of this
distribution (see, for example, Anderson, Guionnet and Zeitouni, \cite{And10}), it follows
from equation (\ref{eqAa}) that
$$
A_\a = \left\{ c \in (-2,2) : \exists \; w \in \C \setminus \R \mbox{ with }
1 - \frac{2 (1 - \a)}{w (w + c)} = \sqrt{ 1 - 4(w + c)^{-2} } \right\},
$$
where we define $\sqrt{r e^{i \t}} := \sqrt{r} e^{i \frac{\t}2}$ for all $r \ge 0$
and $\t \in (-\pi,\pi]$.
Then $A_\a = (-2 \sqrt{\a}, 2 \sqrt{\a})$ and
$\{ w \in \C \setminus \R : 1 - \frac{2 (1 - \a)}{w (w + c)} = \sqrt{ 1 - 4(w + c)^{-2} } \}
= \{ w_{\a,c}, \overline{w_{\a,c}} \}$ for all $c \in (-2 \sqrt{\a}, 2 \sqrt{\a})$, where
$$
w_{\a,c} = \frac{ (1 - \a) c + (1 - \a) \sqrt{ 4 \a - c^2 } \; i}{2 \a}.
$$
The density in Theorem \ref{thSine} is given by
$$
\rho_{\a} (c) = \sqrt{\a} \; \rho_{\text{sc}} \left( \frac{c}{\sqrt{\a}} \right),
$$
for all $c \in (-2 \sqrt{\a}, 2 \sqrt{\a})$, where $\rho_{\text{sc}}$ is the density of the
semi-circle distribution.
%
%

\subsubsection{A measure with two atoms}
\label{secExTwoAtoms}

Fix $\a \in (0,1)$, $\b \in (0,1)$ and $\mu := (1-\b) \d_0 + \b \d_1$.
It follows from equations (\ref{eqAa}) and (\ref{eqCauTra}) that
$$
A_\a = \{ c \in (0,1) : \exists \; w \in \C \setminus \R \mbox{ with }
\a (w + c)^2 + (\b - \a - c) (w + c) + (1 - \b) c = 0 \}.
$$
%
%
The discriminant of the quadratic polynomial is
$(c - c_{\a,\b}^-) (c - c_{\a,\b}^+)$, where
$$
c_{\a,\b}^\pm := \left( \sqrt{(1 - \a) \b} \pm \sqrt{\a (1 - \b)} \right)^2.
$$
%
%
%
Note that $0 < c_{\a,\b}^- < c_{\a,\b}^+ < 1$, and so $A_\a = (c_{\a,\b}^-, c_{\a,\b}^+)$.
Moreover $\{ w \in \C \setminus \R : \a (w + c)^2 + (\b - \a - c) (w + c) + (1 - \b) c = 0 \}
= \{w_{\a,c}, \overline{w_{\a,c}}\}$ for all $c \in (c_{\a,\b}^-, c_{\a,\b}^+)$,
where
$$
w_{\a,c} := - c + \frac{ - \b + \a + c
\pm i \sqrt{ (c - c_{\a,\b}^-) (c_{\a,\b}^+ - c) }}{2 \a}.
$$
The density in Theorem \ref{thSine} is given by
$$
\rho_\a (c)
= \frac{ \sqrt{ (c - c_{\a,\b}^-) \; (c_{\a,\b}^+ - c)}}{2 \pi c (1-c)},
$$
for all $c \in (c_{\a,\b}^-, c_{\a,\b}^+)$.
%
%
This recovers the result of Collins, \cite{Col05}, who took $B_n \in \H_n$
to be a projection of rank $\tilde{q}_n$ with $\frac{\tilde{q}_n}n \to \b \in (0,1)$ as
$n \to \infty$. Collins computed the asymptotics by showing that
$\pi_{q_n} U_n B_n U_n^\ast \pi_{q_n}$ is distributed according to a Jacobi
ensemble of parameters $(q_n,n - \tilde{q}_n - q_n, \tilde{q}_n - q_n)$, and employing
known asymptotic properties of Jacobi polynomials. Another example in which similar
asymptotics arise is the discrete planar bead model examined by Fleming, Forrester,
and Nordenstam, \cite{Fle10}.

\subsubsection{A measure with three atoms}

Fix $\a \in (0,1)$ and $\mu := \frac13 \left( \d_{-1} + \d_0 + \d_1 \right)$.
It follows from equations (\ref{eqAa}) and (\ref{eqCauTra}) that
$$
A_\a = \left\{ c \in (-1,1) :  \exists \; w \in \C \setminus \R \mbox{ with }
\a (w + c)^3 - c (w + c)^2 + \left( \frac23 - \a \right) (w + c) + \frac{c}3 = 0 \right\}.
$$
%
%
The discriminant of the cubic polynomial is
$\frac43 (c^2 - c_\a^-) (c^2 - c_\a^+)$, where
$$
c_\a^\pm := \frac38 \left( g_\a \pm \sqrt{ (g_\a)^2 + \frac{64}3 \left( \frac23 - \a \right)^3 \a } \right),
\hspace{0.5cm}
g_\a := 3 \a^2 + 6 \left( \frac23 - \a \right) \a - \left( \frac23 - \a \right)^2.
$$
%
%
Note that $0 < c_\a^- < c_\a^+ < 1$ when $\a \in (\frac23,1)$,
$c_\a^- = 0$ and $c_\a^+ = 1$ when $\a = \frac23$, and $c_\a^- < 0 < c_\a^+ < 1$ when
$\a \in (0,\frac23)$.
It thus follows that $A_\a = (-\sqrt{c_\a^+}, -\sqrt{c_\a^-})
\bigcup (\sqrt{c_\a^-}, \sqrt{c_\a^+})$ for all $\a \in [\frac23,1)$, and
$A_\a = (-\sqrt{c_\a^+}, \sqrt{c_\a^+})$ for all $\a \in (0,\frac23)$. Moreover,
$\{ w \in \C \setminus \R : \a (w + c)^3 - c (w + c)^2 + \left( \frac23 - \a \right) (w + c)
+ \frac{c}3 = 0 \} = \{w_{\a,c}, \overline{w_{\a,c}} \}$ for all
$c \in A_\a$, where $w_{\a,c}$ is the root of the cubic in the upper half complex plane.

\section{Determinantal structure of Gelfand-Tsetlin patterns}
\label{secDSOGTP}

Define a probability measure on $\overline{\GT_n}$, the set of Gelfand-Tsetlin patterns
of depth $n$, by
\begin{equation}
\label{eqMeasCnCn0}
d\nu_n[y^{(1)},\ldots,y^{(n)}] := \frac1{Z_n}
\det \left[ \phi_i(y_j^{(n)}) \right]_{i,j=1}^n dy^{(n)} dy^{(n-1)} \ldots dy^{(1)},
\end{equation}
for all $(y^{(1)},\ldots,y^{(n)}) \in \overline{\GT_n}$, where $Z_n > 0$ is a normalisation
constant, $d y^{(r)}$ is Lebesgue measure on $\R^r$ for each $r$, and
$\phi_1,\ldots,\phi_n : \R \to \R$ are such that the integrals in Theorem \ref{thmGTDet} are
well-defined and finite. In this section we prove a determinantal structure for the space
$(\overline{\GT_n},\nu_n)$. Though the main result of this section, Theorem \ref{thmGTDet},
can be deduced from the more general results of Defosseux, \cite{Man08}, we give a simplified
proof with an alternative expression for the correlation kernel.

\begin{rem}
\label{remphiEx}
The measure in equation (\ref{eqUniMeasGTn}) can be written in the above form by taking
$\phi_i = \d_{x_i^{(n)}}$ for all $i \in \{1,\ldots,n\}$. As we shall see in
Section (\ref{secUIE}), the measure induced by the eigenvalue minor process of UIEs can
also be written in this form.
\end{rem}

For technical reasons we consider a subset of $\overline{GT}_n$ on which the measure in
equation (\ref{eqMeasCnCn0}) is supported. We say that a pair
$(y^{(r)},y^{(r+1)}) \in \CC_r \times \CC_{r+1}$ is {\em asymmetrically interlaced} if 
$$
y^{(r+1)}_1 > y^{(r)}_1 \ge y^{(r+1)}_2 > y^{(r)}_2 > \cdots > y^{(r)}_r \ge y^{(r+1)}_{r+1}.
$$
We denote this by $y^{(r+1)} \succ y^{(r)}$. Also, for each $n \ge 1$, define
$\GT_n \subset \CC_1 \times \cdots \times \CC_n$ by
$$
\GT_n := \left\{ (y^{(1)},\ldots,y^{(n)}) \in \CC_1 \times \cdots \times \CC_n :
y^{(n)} \succ y^{(n-1)} \succ \cdots \succ y^{(1)} \right\}.
$$
Comparing with $\overline{\GT_n}$ (see equation (\ref{eqGTnS})), $\GT_n \subset \overline{\GT_n}$
is the set of Gelfand-Tsetlin patterns of depth $n$ with distinct particles and for which
particles on neighbouring levels satisfy the asymmetric interlacing constraint. It is
easy to see that $\nu_n$ is supported on $\GT_n$.

As in Section \ref{secTEOTMOTGUE}, we identify $\GT_n$ with a space of configurations of
$\frac12 n(n+1)$ particles on $\{1,\ldots,n\} \times \R$. Definition \ref{defRanPoiFie}
thus implies that $(\GT_n,\nu_n)$ is a random point field on
$\{1,\ldots,n\} \times \R$. We shall prove the following:
\begin{thm}
\label{thmGTDet}
Define $\Phi_n : \R^n \to \R$ by
\begin{equation}
\label{eqPhin}
\Phi_n(y) := \left( \prod_{k=1}^n \phi_k(y_k) \right) \Delta_n(y),
\end{equation}
for $y \in \R^n$. Also define $B_n := \int_{\R^n} dy \; \Phi_n(y)$. Finally, letting $\S_n$
be the set of permutations of $\{1,\ldots,n\}$, define $\S_n \CC_n := \bigcup_{\s \in \S_n} \s (\CC_n)$.
Then $B_n \neq 0$, and the random point field $(\GT_n,\nu_n)$ is determinantal
with correlation kernel $K_n : (\{1,\ldots,n\} \times \R)^2 \to \R$ which satisfies
\begin{eqnarray*}
K_n ((r,u),(s,v))
& = & \frac1{B_n} \int_{\S_n \CC_n} dy \; \Phi_n(y)
\sum_{j=1}^n 1_{v \le u < y_j} \frac{(y_j - u)^{n-r-1}}{(n-r-1)!}
\frac{\partial^{n-s}}{\partial v^{n-s}} \prod_{i \neq j} \left( \frac{v - y_i}{y_j - y_i} \right) \\
& - & \frac1{B_n} \int_{\S_n \CC_n} dy \; \Phi_n(y)
\sum_{j=1}^n 1_{v > u \ge y_j} \frac{(y_j - u)^{n-r-1}}{(n-r-1)!}
\frac{\partial^{n-s}}{\partial v^{n-s}} \prod_{i \neq j} \left( \frac{v - y_i}{y_j - y_i} \right),
\end{eqnarray*}
for all $r \in \{1,\ldots,n-1\}$, $s \in \{1,\ldots,n\}$ and $u,v \in \R$.
\end{thm}

In order to show this we consider a related measure on systems of interlaced
particles with the same number of {\em indistinguishable} particles on each level. 
Given $z,z' \in \S_n \CC_n$ with $z \in \s^{-1} (\CC_n)$ and $z' \in \tau^{-1} (\CC_n)$
some $\s,\tau \in \S_n$, we say that the pair $(z,z')$ is {\em interlaced} if
$$
z'_{\tau(1)} > z_{\s(1)} \ge z'_{\tau(2)} > z_{\s(2)} \ge \cdots \ge z'_{\tau(n)} > z_{\s(n)}.
$$
Let $F_n \subset (\S_n \CC_n)^2$ be the set of all interlaced pairs. A nice characterisation
of this type of interlacing is given in Warren, \cite{W06}: Given $z,z' \in \S_n \CC_n$ with
$z \in \s^{-1} (\CC_n)$ and $z' \in \tau^{-1} (\CC_n)$ some $\s,\tau \in \S_n$,
\begin{equation}
\label{eqWar06}
1_{(z,z') \in F_n} = \det \left[ 1_{z'_{\tau(k)} > z_{\s(j)}} \right]_{j,k=1}^n.
\end{equation}
%
%

Fix $n \in \N$, $M>0$ and $(c_1,\ldots,c_n) \in \CC_n$ with $c_1 = -M$. Consider
the space $(\R^n)^n$, interpreted
as the set of configurations of $n^2$ particles in $\R^n$ with exactly $n$ particles in each
$\R$. Denoting elements of this space by $\bar z := (z^{(1)},\ldots,z^{(n)})$, 
let $E \subset (\R^n)^n$ be the set of configurations for which
\begin{itemize}
\item $M \ge z_j^{(n)} \ge -M$ for all $j \in \{1,\ldots,n\}$,
\item $z^{(1)}_{\tau(j)} = c_j$ for all $j \in \{2,\ldots,n\}$ whenever
$z^{(1)} \in \tau^{-1} (\CC_n)$ some $\tau \in \S_n$,
\item $(z^{(r)},z^{(r+1)}) \in F_n$ for all $r \in \{1,\ldots,n-1\}$.
\end{itemize}
Choosing $M>0$ sufficiently large, we can define the measure, $\xi_n$, on $(\R^n)^n$ by
\begin{equation}
\label{eqMeasCnCn1}
d\xi_n [\bar z] := \frac1{(n!)^n Z}
\begin{cases}
\det \left[ \phi_i(z_{\s(j)}^{(n)}) \right]_{i,j=1}^n dz^{(n)} dz^{(n-1)} \ldots dz^{(1)}
& ; \bar{z} \in E \mbox{ with } z^{(n)} \in \s^{-1} (\CC_n), \\
0 & ; \bar{z} \in (\R^n)^n \setminus E.
\end{cases}
\end{equation}
where $Z > 0$ is a normalisation constant, and $dz^{(r)}$ is the Lebesgue measure on
$\R^n$ for each $r$.

We identify $(\R^n)^n$ with a space of configurations of $n^2$ particles on
$\{1,\ldots,n\} \times \R$ using the natural map from $(\R^n)^n$ to
$(\{1,\ldots,n\} \times \R)^{n^2}$ given by
$$
(z^{(1)},\ldots,z^{(n)}) \mapsto \left( (1,z_1^{(1)}),\ldots,(1,z_n^{(1)}),
(2,z_1^{(2)}),\ldots,(2,z_n^{(2)}),(3,z_1^{(3)}),\ldots,(3,z_n^{(3)}),\ldots\ldots \right),
$$
for all $(z^{(1)},\ldots,z^{(n)}) \in (\R^n)^n$. In words, the first $n$ particles of each
configuration are contained in $\{1\} \times \R$, the next $n$ particles are contained in
$\{2\} \times \R$, the next $n$ in $\{3\} \times \R$ etc. Definition \ref{defRanPoiFie} thus
implies that $((\R^n)^n,\xi_n)$ is a random point field on $\{1,\ldots,n\} \times \R$. We
now show this field is determinantal and calculate the correlation kernel.

\begin{lem}
\label{lemMeasdetGT}
Define $\phi_{0,1} : \{1,\ldots,n\} \times \R \to \R$,
$\phi_{1,2},\phi_{2,3}, \ldots, \phi_{n-1,n} : \R \times \R \to \R$ and
$\phi_{n,n+1} : \R \times \{1,\ldots,n\} \to \R$ by
\begin{eqnarray*}
\phi_{0,1} (i,v) & := &
\begin{cases}
1_{M \ge v \ge -M} & ; i = 1, \\
\d_{c_i} (v) & ; i \in \{2,\ldots,n\},
\end{cases} \\
\phi_{r,r+1} (u,v) & := & 1_{M \ge v > u \ge c_n} \; \mbox{ for all } r \in \{1,\ldots,n-1\}, \\
\phi_{n,n+1} (u,j) & := & \phi_j(u) 1_{M \ge u \ge -M}.
\end{eqnarray*}
Also define $z^{(0)}_j = z^{(n+1)}_j := j$ for all $j \in \{1,\ldots,n\}$. Then for all
$\bar z \in (\R^n)^n$,
\begin{equation}
\label{eqMeasCnCn2}
d\xi_n [\bar z] = \frac1{(n!)^n Z} \prod_{r=0}^n
\det[ \phi_{r,r+1} (z^{(r)}_j,z^{(r+1)}_k) ]_{j,k=1}^n dz^{(n)} dz^{(n-1)} \ldots dz^{(1)}.
\end{equation}
\end{lem}

\begin{proof}
Fixing $\bar z \in (\CC_n)^n$, it follows from the definition $E$ that,
$$
1_E (\bar z)
= 1_{M \ge z_1^{(1)} \ge -M} \left( \prod_{j=2}^n \d_{c_j} (z_j^{(1)}) \right)
\left( \prod_{r=1}^{n-1} 1_{(z^{(r)},z^{(r+1)}) \in F_n} \right)
\left( \prod_{i=1}^n 1_{M \ge z_i^{(n)} \ge -M} \right).
$$
%
%
The interlacing formula of Warren (see equation (\ref{eqWar06})) thus gives
$$
1_E (\bar z)
= 1_{M \ge z_1^{(1)} \ge -M} \left( \prod_{j=2}^n \d_{c_j} (z_j^{(1)}) \right)
\left( \prod_{r=1}^{n-1} \det \left[ 1_{M \ge z^{(r+1)}_k > z^{(r)}_j \ge c_n} \right]_{j,k=1}^n \right)
\left( \prod_{i=1}^n 1_{M \ge z_i^{(n)} \ge -M} \right).
$$
%
%
The required result in this case follows from equation (\ref{eqMeasCnCn1}).
The result when $\bar z \in (\S_n \CC_n)^n$ follows since the expressions given in
equations (\ref{eqMeasCnCn1}) and (\ref{eqMeasCnCn2}) are invariant under permutations.
Finally, the result is trivially true when $\bar z \in (\R^n)^n \setminus (\S_n \CC_n)^n$,
since both expressions are identically $0$.
\end{proof}

Equation (\ref{eqMeasCnCn2}) gives
$$
Z = \frac1{(n!)^n} \int_{(\R^n)^n}
\prod_{r=0}^n \det[ \phi_{r,r+1} (z^{(r)}_j,z^{(r+1)}_k) ]_{j,k=1}^n
dz^{(n)} dz^{(n-1)} \ldots dz^{(1)}.
$$
The Cauchy-Binet identity (Proposition 2.10 of Johansson, \cite{Joh06b}) then gives
$Z = \det A$, where $A \in \C^{n \times n}$ is given by
$A_{ij} := \phi_{0,n+1} (i,j)$ for all $i,j \in \{1,\ldots,n\}$, and
$\phi_{0,s} : \{1,\ldots,n\} \times \R \to \R$, $\phi_{r,s} : \R \times \R \to \R$ and
$\phi_{r,n+1} : \R \times \{1,\ldots,n\} \to \R$
are defined by
$$
\phi_{r,s} (u,v) := 1_{s>r} \int_\R dz_1 \cdots \int_\R dz_{s-r-1} \;
\phi_{r,r+1} (u,z_1) \phi_{r+1,r+2} (z_1,z_2) \ldots \phi_{s-1,s} (z_{s-r-1},v),
$$
for all $r,s \in \{1,\ldots,n\}$. Therefore
\begin{eqnarray}
\label{eqphirs}
\phi_{r,s} (u,v)
& = & \frac{(v-u)^{s-r-1}}{(s-r-1)!} 1_{M \ge v > u \ge c_n} 1_{s>r}, \\
\label{eqphi0s}
\phi_{0,s} (i,v)
& = & \frac{(v - c_i)^{s-2+1_{i=1}}}{(s-2+1_{i=1})!} 1_{M \ge v > c_i}, \\
\label{eqphirn+1}
\phi_{r,n+1} (u,j)
& = & \int_{-M}^M dz \; \phi_j(z) \frac{(z-u)^{n-r-1}}{(n-r-1)!} 1_{z > u \ge c_n} 1_{r \le n-1}
+ \phi_j(u) 1_{M \ge u \ge -M} 1_{r=n}, \\
\label{eqphi0n+1}
A_{ij}
& = & \int_{-M}^M dz \; \phi_j(z) \frac{(z - c_i)^{n-2+1_{i=1}}}{(n-2+1_{i=1})!},
\end{eqnarray}
for all $r,s,i,j \in \{1,\ldots,n\}$ and $u,v \in \R$.

\begin{prop}
\label{prCnCnDet}
Letting $\Phi_n : \R^n \to \R$ be that given in equation (\ref{eqPhin}), and $M > 0$ be that
given in equation (\ref{eqMeasCnCn1}), define $B_n^{(M)} := \int_{[-M,M]^n} dz \; \Phi_n(z)$.
Then $B_n^{(M)} > 0$, and the random point field $((\R^n)^n,\xi_n)$ is determinantal with
correlation kernel $J_n : (\{1,\ldots,n\} \times \R)^2 \to \R$, which satisfies
\begin{eqnarray*}
J_n ((r,u),(s,v))
& = & \frac1{B_n^{(M)}} \int_{\S_n \CC_n^{(M)}} dz \; \Phi_n(z)
\sum_{j=1}^n 1_{v \le u < z_j} \frac{(z_j-u)^{n-r-1}}{(n-r-1)!}
\frac{\partial^{n-s}}{\partial v^{n-s}} \prod_{i \neq j} \left( \frac{v - z_i}{z_j - z_i} \right) \\
& - & \frac1{B_n^{(M)}} \int_{\S_n \CC_n^{(M)}} dz \; \Phi_n(z)
\sum_{j=1}^n 1_{v > u \ge z_j} \frac{(z_j-u)^{n-r-1}}{(n-r-1)!}
\frac{\partial^{n-s}}{\partial v^{n-s}} \prod_{i \neq j} \left( \frac{v - z_i}{z_j - z_i} \right),
\end{eqnarray*}
for all $r \in \{1,\ldots,n-1\}$, $s \in \{1,\ldots,n\}$ and $u,v \in (-M,M)$, where
$\CC_n^{(M)} := \CC_n \cap [-M,M]^n$.
\end{prop}

\begin{proof}
The fact that $((\R^n)^n,\xi_n)$ is determinantal follows from Lemma
\ref{lemMeasdetGT} and Proposition 2.13 of Johansson, \cite{Joh06b}. A correlation
kernel is given by
\begin{equation}
\label{eqprCnCnDet1}
J_n ((r,u),(s,v)) = - \phi_{r,s} (u,v) + \tilde J_n ((r,u),(s,v)),
\end{equation}
for all $r,s \in \{1,\ldots,n\}$ and $u,v \in \R$, where
\begin{eqnarray*}
\tilde J_n ((r,u),(s,v))
& = & \sum_{i,j=1}^n (-1)^{i+j} \phi_{0,s} (i,v)
\frac{\det A (i,j)}{\det A} \phi_{r,n+1} (u,j),
\end{eqnarray*}
and $A(i,j) \in \C^{(n-1) \times (n-1)}$ is the sub-matrix of $A$ obtained by
removing row $i$ and column $j$.

Fix $r \in \{1,\ldots,n-1\}$, $s \in \{1,\ldots,n\}$ and $u,v \in (-M,M)$. First note equation
(\ref{eqphirs}) gives
\begin{equation}
\label{eqphirs2}
\phi_{r,s} (u,v)
= 1_{v > u} \frac{\partial^{n-s}}{\partial v^{n-s}} \frac{(v - u)^{n-r-1}}{(n-r-1)!}
= 1_{v > u} \frac{\partial^{n-s}}{\partial v^{n-s}}
\sum_{j=1}^n \frac{(z_j - u)^{n-r-1}}{(n-r-1)!}
\prod_{i \neq j} \left( \frac{v - z_i}{z_j - z_i} \right),
\end{equation}
for any $z \in \S_n \CC_n$, where the last step follows from Lagrange interpolation.
Also equations (\ref{eqphi0s}) and (\ref{eqphirn+1}) give
\begin{equation}
\label{eqtildeJn}
\tilde J_n ((r,u),(s,v))
= \frac{\partial^{n-s}}{\partial v^{n-s}}
\sum_{j=1}^n \frac{\det A^{(j,v)}}{\det A}
\int_{-M}^M dz_j \; \phi_j(z_j) \frac{(z_j-u)^{n-r-1}}{(n-r-1)!} 1_{z_j>u},
\end{equation}
where $A^{(j,v)} \in \C^{n \times n}$ is $A$ with column $j$ replaced by
$(\phi_{0,n} (1,v),\ldots,\phi_{0,n} (n,v))^T$. This can be verified by taking
a cofactor expansion of $\det A^{(j,v)}$ along column $j$.
%
%
Moreover equation (\ref{eqphi0n+1}) gives
$$
\det A
= \int_{[-M,M]^n} dz \; \sum_{l_1=0}^{n-1} \; \sum_{l_2,\ldots,l_n=0}^{n-2} 
\left( \prod_{k=1}^n \frac{ \phi_k(z_k) \; (-c_k)^{n-2+1_{k=1}-l_k}}
{(l_k)! \; (n-2+1_{k=1}-l_k)!} \right) \det \left[ (z_j)^{l_i} \right]_{i,j=1}^n.
$$
%
%
The only non-zero terms in the above sum are those for which $l_1,\ldots,l_n$ are
distinct. Equation (\ref{eqVanDet}) then gives
$$
\det A = \frac1{(n-1)!} \left( \prod_{k=0}^{n-2} \frac1{k!} \right)^2
\Delta_{n-1}(c_2,\ldots,c_n) \; B_n^{(M)},
$$
where $B_n^{(M)}$ is defined in the statement of the Proposition.
Therefore $B_n^{(M)} > 0$, since $(c_1,\ldots,c_n) \in \CC_n$ and $\det A = Z$, where $Z > 0$
is the normalisation constant in equation (\ref{eqMeasCnCn1}). Similarly
$$
\det A^{(j,v)} = \frac1{(n-1)!} \left( \prod_{k=0}^{n-2} \frac1{k!} \right)^2
\Delta_{n-1}(c_2,\ldots,c_n)
\int_{[-M,M]^{n-1}} \left( \prod_{k \neq j} \phi_k(z_k) \; dz_k \right) \Delta_n(z^{(j,v)}),
$$
for all $j \in \{1,\ldots,n\}$, where $z^{(j,v)} := (z_1,\ldots,z_{j-1},v,z_{j+1},\ldots,z_n)$.
Equations (\ref{eqVanDet}), (\ref{eqPhin}) and (\ref{eqtildeJn}) thus give
$$
\tilde J_n ((r,u),(s,v))
= \frac1{B_n^{(M)}} \frac{\partial^{n-s}}{\partial v^{n-s}}
\int_{\S_n \CC_n^{(M)}} dz \; \Phi_n(z) \sum_{j; z_j > u} \frac{(z_j-u)^{n-r-1}}{(n-r-1)!}
\prod_{i \neq j} \left( \frac{v - z_i}{z_j - z_i} \right).
$$
%
%
Equations (\ref{eqprCnCnDet1}) and (\ref{eqphirs2}) then give the required result.
%
%
\end{proof}

We are now in a position to give a proof of Theorem \ref{thmGTDet}:
\begin{proof}[Proof of Theorem \ref{thmGTDet}]
For each $(c_1,\ldots,c_n) \in \CC_n$ with $c_1 = - M < 0$, using superscripts to emphasise
the dependence on $(c_1,\ldots,c_n)$, Proposition \ref{prCnCnDet} implies that
$((\R^n)^n,\xi_n^{(c_1,\ldots,c_n)})$ is a determinantal random point field with correlation
kernel $J_n^{(c_1,\ldots,c_n)} : (\{1,\ldots,n\} \times \R)^2 \to \R$.
When restricted to the domain $(\{1,\ldots,n-1\} \times (-M,M)) \times
(\{1,\ldots,n\} \times (-M,M))$, this kernel depends on $M$ and does not depend on $c_2,\ldots,c_n$.
Also it follows from equations (\ref{eqMeasCnCn0}) and (\ref{eqMeasCnCn1}) that
$\xi_n^{(c_1,\ldots,c_n)}$ induces the probability measure on $\GT_n$ given by
$$
\nu_n^{(M)} [A]
:= \frac{\nu_n [ A \cap ( \CC_1^{(M)} \times \cdots \times \CC_n^{(M)} ) ]}
{\nu_n [ \GT_n \cap ( \CC_1^{(M)} \times \cdots \times \CC_n^{(M)} ) ]},
$$
for all $A \subset \GT_n$ measurable, where $\CC_r^{(M)} = \CC_r \cap [-M,M]^r$
for all $r \in \{1,\ldots,n\}$.
%
%
Therefore $(\GT_n,\nu_n^{(M)})$ is a determinantal random point field with correlation kernel
$K_n^{(M)} : (\{1,\ldots,n\} \times (-M,M))^2 \to \R$ which satisfies
$K_n^{(M)} = J_n^{(c_1,\ldots,c_n)}$ in the domain $(\{1,\ldots,n-1\} \times (-M,M)) \times
(\{1,\ldots,n\} \times (-M,M))$. The required result follows by letting $M \to \infty$.
\end{proof}

We finish this section by obtaining useful contour integral expressions for the kernel in
Theorem \ref{thmGTDet}:
\begin{prop}
\label{prIntExpKl}
For all $r \in \{1,\ldots,n-2\}$, $s \in \{1,\ldots,n\}$, and $u,v \in \R$,
\begin{eqnarray*}
\lefteqn{K_n ((r,u),(s,v))
\; = \; \frac1{(2\pi)^2} \frac{(n-s)!}{(n-r-1)!}  \frac1{B_n} \int_{\S_n \CC_n} dy \;
\Phi_n(y) \times} \\
& \times & \int_{\g(u,v,y)} dw \int_{\G(u,v,y)} dz \; \frac{(z - u)^{n-r-1}}{(w - v)^{n-s+1}}
\frac1{w-z} \prod_{i=1}^n \left( \frac{w - y_i}{z - y_i} \right).
\end{eqnarray*}
Here $\g(u,v,y)$ is a counter-clockwise simple closed contour around $v$. Whenever $v \le u$,
$\G(u,v,y)$ is a clockwise simple closed contour which passes through $u$, contains
$\{y_j : y_j > u \}$ and does not contain $\{y_j: y_j < u\}$. Whenever $v > u$,
$\G(u,v,y)$ is a counter-clockwise simple closed contour which passes through $u$, contains
$\{y_j : y_j < u \}$ and does not contain $\{y_j: y_j > u\}$. Finally the contours do not intersect.
This holds with the understanding that $(z - u)^{n-r-1} \prod_i (\frac1{z - y_i}) = (z - u)^{n-r-2}
\prod_{i \neq k} (\frac1{z - y_i})$ whenever $u=y_k$ for some $k \in \{1,\ldots,n\}$.

Also for all $r \in \{1,\ldots,n-1\}$ and $u,v \in \R$,
\begin{eqnarray*}
\lefteqn{K_n ((r,u),(r,v))
\; = \; \frac1{(2\pi)^2} \frac1{B_n} \int_{\S_n \CC_n} dy \;
\Phi_n(y) \times} \\
& \times & \int_\g dw \int_{\G(u,v,y)} dz
\left( \frac{(z + v - u)^{n-r} - z^{n-r}}{(v - u) w^{n-r+1}} \right)
\sum_{j=1}^n \frac{v - y_j}{(z + v - y_j)^2} \prod_{i \neq j}
\left( \frac{w + v - y_i}{z + v - y_i} \right).
\end{eqnarray*}
Here $\g$ is a counter-clockwise simple closed contour around $0$. Whenever $v \le u$,
$\G(u,v,y)$ is a clockwise simple closed contour in
$\C \setminus \{y_1 - v, \ldots, y_n - v\}$ which contains $\{y_j - v : y_j > u \}$ and does not
contain $\{y_j - v : y_j \le u\}$. Whenever $v > u$, $\G(u,v,y)$ is a counter-clockwise simple closed
contour in $\C \setminus \{y_1 - v, \ldots, y_n - v\}$ which contains $\{y_j - v : y_j \le u \}$ and
does not contain $\{y_j - v : y_j > u\}$. This holds with the understanding that
$\frac{(z + v - u)^{n-r} - z^{n-r}}{v - u} = (n-r) z^{n-r-1}$ whenever $u = v$.
\end{prop}

\begin{proof}
For all $r \in \{1,\ldots,n-1\}$, $s \in \{1,\ldots,n\}$ and $u,v \in \R$,
Theorem (\ref{thSine}) gives
\begin{equation}
\label{eqprIntExpKl1}
K_n ((r,u),(s,v)) = \frac1{(2\pi)^2} \frac{(n-s)!}{(n-r-1)!}  \frac1{B_n} \int_{\S_n \CC_n} dy \;
\Phi_n(y) \; G_{r,s}^{(u,v)} (y),
\end{equation}
where $G_{r,s}^{(u,v)} : \S_n \CC_n \to \R$ is given by
\begin{eqnarray}
\nonumber
G_{r,s}^{(u,v)} (y)
& := & (2\pi)^2 \sum_{j=1}^n 1_{v \le u < y_j} (y_j - u)^{n-r-1}
e_{s-1} \left( v - y_1, \p{v - y_j}, v - y_n \right) \prod_{i \neq j} \left( \frac1{y_j - y_i} \right) \\
\label{eqGrsuvy0}
& - & (2\pi)^2 \sum_{j=1}^n 1_{v > u \ge y_j} (y_j - u)^{n-r-1}
e_{s-1} \left( v - y_1, \p{v - y_j}, v - y_n \right) \prod_{i \neq j} \left( \frac1{y_j - y_i} \right),
\end{eqnarray}
for all $y \in \S_n \CC_n$. Here $e_{s-1}$ is the elementary symmetric polynomial of degree $s-1$.
%
%
Then, whenever $r \le n-2$, the residue Theorem gives the first part of the result.
%
%
To see the second part note that the residue Theorem alternatively gives
$$
G_{r,r}^{(u,v)}(y)
= \int_{\g(u,v,y)} dw \int_{\G(u,v,y)} dz \; \frac{(z + v - u)^{n-r-1}}{w^{n-r+1}}
\frac1{w-z} \prod_{i=1}^n \left( \frac{w + v - y_i}{z + v - y_i} \right),
$$
for all $r \in \{1,\ldots,n-1\}$, $u,v \in \R$, and $y \in \S_n \CC_n$, where we choose
the contours so that they do not intersect, $\g(u,v,y)$ is not in the interior of
$\G(u,v,y)$, $\g(u,v,y)$ is a counter-clockwise simple closed contour around $0$, and
$\G(u,v,y)$ is chosen as in the second part of the result.
%
%
Fixing $r \in \{1,\ldots,n-1\}$, $u,v \in \R$, and $y \in \S_n \CC_n$, write
\begin{equation}
\label{eqprIntExpKl2}
G_{r,r}^{(u,v)}(y) = \sum_{m=1}^{n-r} \binom{n-r}{m-1} (v-u)^{m-1} F_m,
\end{equation}
where
$$
F_m := \int_{\g(u,v,y)} dw \int_{\G(u,v,y)} dz \; \frac{z^{n-r-m}}{w^{n-r+1}}
\frac1{w-z} \prod_{i=1}^n \left( \frac{w + v - y_i}{z + v - y_i} \right).
$$
Note, for all $b \in \R$ sufficiently close to $1$, the residue Theorem implies that
$\g(u,v,y)$ and $\G(u,v,y)$ can be replaced by $b \; \g(u,v,y)$ and
$b \; \G(u,v,y)$ respectively, and so
$$
F_m = b^{-m} \int_{\g(u,v,y)} dw \int_{\G(u,v,y)} dz \; \frac{z^{n-r-m}}{w^{n-r+1}}
\frac1{w-z} \prod_{i=1}^n \left( \frac{bw + v - y_i}{bz + v - y_i} \right).
$$
%
%
Differentiate both sides with respect to $b$ and set $b=1$ to get
$$
F_m = \frac1m \int_{\g(u,v,y)} dw \int_{\G(u,v,y)} dz \;
\frac{z^{n-r-m}}{w^{n-r+1}} \sum_{j=1}^n \frac{v - y_j}{(z + v - y_j)^2}
\prod_{i \neq j} \left( \frac{w + v - y_i}{z + v - y_i} \right).
$$
%
%
The residue Theorem implies that $\g(u,v,y)$ can be replaced
by any counter-clockwise simple closed contour around $0$, $\g$.
Equation (\ref{eqprIntExpKl2}) then gives
$$
G_{r,r}^{(u,v)}(y) = \sum_{m=1}^{n-r} \binom{n-r}m \frac{(v-u)^{m-1}}{n-r}
\int_\g dw \int_{\G(u,v,y)} dz \; \frac{z^{n-r-m}}{w^{n-r+1}}
\sum_{j=1}^n \frac{v - y_j}{(z + v - y_j)^2} \prod_{i \neq j} \left( \frac{w + v - y_i}{z + v - y_i} \right).
$$
%
%
This holds for all $r \in \{1,\ldots,n-1\}$, $u,v \in \R$, and $y \in \S_n \CC_n$.
Equation (\ref{eqprIntExpKl1}) gives the required result.
%
%
\end{proof}

\section{Proof of Theorem \ref{thSine}}
\label{secPOTMR}

In this section we prove Theorem \ref{thSine}. Fix $a,b \in \R$ with $a<b$. For each
$n \in \N$, choose $q_n \in \{1,\ldots,n\}$ and $x^{(n)} \in \CC_n \cap [a,b]^n$ as in
sections \ref{secfp} and \ref{secR}, and equip $\GT_n$ with the measure given in equation
(\ref{eqUniMeasGTn}). This satisfies equation (\ref{eqMeasCnCn0})
with $\phi_i = \d_{x_i^{(n)}}$ for all $i \in \{1,\ldots,n\}$. Let $K_n : (\{1,\ldots,n\}
\times [a,b])^2 \to \C$ be the associated correlation kernel given equation in
(\ref{eqKnrusvFixTopLine}).

Assume hypothesis \ref{hypWeakConv}. Fix $c \in A_\a$ and $U,V \subset \R$ compact, where
$A_\a \subset (a,b)$ is given in equation (\ref{eqAa}). Proposition \ref{prIntExpKl} gives
\begin{eqnarray}
\label{eqIntExpCorKer}
& & \hspace{2.5cm}  \frac{4 \pi^2}n 
K_n \left( \left( q_n, c + \frac{u}n \right), \left( q_n, c + \frac{v}n \right) \right) = \\
\nonumber
& & \int_{\g_n} dw \int_{\G_n} dz
\left( \frac{(z + \frac{v - u}n)^{n-q_n} - z^{n-q_n}}{ (v-u) w^{n-q_n+1}} \right)
\sum_{j=1}^n \frac{c + \frac{v}n - x_j^{(n)}}{(z + c + \frac{v}n - x_j^{(n)})^2}
\prod_{i \neq j} \left( \frac{w + c + \frac{v}n - x_i^{(n)}}{z + c + \frac{v}n - x_i^{(n)}} \right),
\end{eqnarray}
for all $n$ sufficiently large, $u \in U$ and $v \in V$, where
$\g_n$ is a counter-clockwise simple closed contour around $0$, and
$\G_n$ is a simple closed contour in
$\C \setminus \{x_1^{(n)} - \frac{v}n - c, \ldots, x_n^{(n)} - \frac{v}n - c \}$ which
satisfies
\begin{itemize}
\item
Whenever $v \le u$, $\G_n$ is clockwise, contains
$\{x_j^{(n)} - \frac{v}n - c : x_j^{(n)} > \frac{u}n + c \}$ and does not contain
$\{x_j^{(n)} - \frac{v}n - c : x_j^{(n)} \le \frac{u}n + c \}$.
\item
Whenever $v > u$, $\G_n$ is counter-clockwise, contains
$\{x_j^{(n)} - \frac{v}n - c : x_j^{(n)} \le \frac{u}n + c \}$ and does not contain
$\{x_j^{(n)} - \frac{v}n - c : x_j^{(n)} > \frac{u}n + c \}$.
\end{itemize}

We examine the asymptotics of this kernel via saddle point analysis. First note, for all $n$
sufficiently large, $u \in U$, $v \in V$, and $z,w \in \C \setminus \R$, the integrand can be
rewritten as
\begin{equation}
\label{eqIntExp}
n \left( \frac{(1 + \frac{v - u}{n z})^{n-q_n} - 1}{v-u} \right) g_{n,v}(w,z)
\; e^{n(h_{n,v}(w) - h_{n,v}(z))},
\end{equation}
where, using the principal value of the logarithm,
$h_{n,v} : \C \setminus \R \to \C$ and $g_{n,v} : (\C \setminus \R)^2 \to \C$ are given by
\begin{eqnarray}
\label{eqhn}
h_{n,v} (w) & := & \int \log(w + c - x) \mu_{n,v} [dx] - \frac{n-q_n}n \log(w), \\
\label{eqhngn}
g_{n,v} (w,z) & := & \left\{
\begin{array}{rl}
\frac{h_{n,v}'(w) - h_{n,v}'(z)}{w - z} + \frac{h_{n,v}'(w)}w; & w \neq z, \\
h_{n,v}''(w) + \frac{h_{n,v}'(w)}w; & w=z,
\end{array}
\right.
\end{eqnarray}
and $\mu_{n,v}$ is the empirical probability measure
\begin{equation}
\label{eqEmpProbMeas}
\mu_{n,v} := \frac1n \sum_{i=1}^n \d_{x_i^{(n)} - \frac{v}n}.
\end{equation}
The following Lemma proves the existence of appropriate saddle points of $h_{n,v}$ for the analysis,
and the first part of Theorem \ref{thSine}. 

\begin{lem}
\label{lemcp}
Define $h : \C \setminus \R \to \C$ by
\begin{equation}
\label{eqh0}
h(w) := \int \log(w + c - x) \mu[dx] - (1 - \a) \log(w),
\end{equation}
for all $w \in \C \setminus \R$. Then there exists a $w_0 \in \C$ with $\Im(w_0) > 0$
and $\{ w \in \C \setminus \R : h'(w) = 0 \} = \left\{ w_0, \overline{w_0} \right\}$.
Also $h''(w_0) \neq 0$. Moreover, given $v \in V$ and $n$ sufficiently large,
there exists a $w_{n,v} \in \C$ with $\Im(w_{n,v}) > 0$ and
$\{ w \in \C \setminus \R : h_{n,v}'(w) = 0 \} = \left\{ w_{n,v}, \overline{w_{n,v}} \right\}$.
Finally
$$
\lim_{n \to \infty} \sup_{v \in V} |w_{n,v} - w_0| = 0.
$$
\end{lem}

\begin{proof}
Since roots of $h'$ and $h_{n,v}'$ occur in complex conjugate pairs,
we shall restrict our attention to $\{w \in \C : \Im(w) > 0 \}$.
Equations (\ref{eqhn}) and (\ref{eqEmpProbMeas}) give
$$
n w \prod_{i=1}^n (w + c + \frac{v}n - x_j^{(n)}) h_{n,v}'(w)
= w \sum_{j=1}^n \prod_{i \neq j} (w + c + \frac{v}n - x_i^{(n)})
- (n - q_n) \prod_{i=1}^n (w + c + \frac{v}n - x_j^{(n)}),
$$
for all $n$ sufficiently large, $v \in V$ and $w \in \C$ with $\Im(w) > 0$.
%
%
The right hand side, a polynomial of degree $n$ with real coefficients,
has at least $n-2$ roots in $\R$.
%
%
%
Thus $h_{n,v}'$ has at most one root (counting multiplicities) in $\{w \in \C : \Im(w) > 0\}$.

Since $c \in A_\a$, equations (\ref{eqAa}) and (\ref{eqh0}) imply that
$h'$ has at least one root in $\{ w \in \C: \Im(w) > 0 \}$. Denoting this by $w_0$,
we now show that, for any $\e \in (0,\Im(w_0))$ and $j \ge 0$
\begin{equation}
\label{eqlemcp1}
\lim_{n \to \infty} \sup_{v \in V} \sup_{w \in \bar{B}(w_0,\e)} | h_{n,v}^{(j)}(w) - h^{(j)}(w) | = 0.
\end{equation}
We use the method of contradictions to prove the result for $j=0$.
Assume that this does not hold for some $\e \in (0,\Im(w_0))$.
Thus there exists some $\xi > 0$ for which, for all $n \ge 1$,
there exists some $m_n \ge n$ and $z_n \in \bar{B}(w_0, \e)$ with
$\xi \le \sup_{v \in V} | h_{m_n,v}(z_n) - h(z_n) |$.
%
%
Also the Bolzano-Weierstrass Theorem implies that we can choose
$\{z_n\}_{n\ge1}$ to be convergent. Denoting the limit by $z_0$,
$$
\xi \le \sup_{v \in V} |h_{m_n,v}(z_n) - h_{m_n,v}(z_0)| + \sup_{v \in V} |h_{m_n,v}(z_0) - h(z_0)|
+ |h(z_0) - h(z_n)|,
$$
for all $n$ sufficiently large. Finally note equation (\ref{eqhn}) gives
$\sup \{ |h_{m_n,v}'(w)| : v \in V \mbox{ and } w \in \bar{B}(w_0,\e) \}
\le 2 (|\Im(w_0)| - \e)^{-1}$ for all $n$ sufficiently large,
and so
$$
\sup_{v \in V} |h_{m_n,v}(z_0) - h(z_0)| \ge \frac{\xi}2
$$
for all $n$ sufficiently large.
%
%
However, since $\frac{q_n}n \to \a$ and $\mu_{n,0}
\to \mu$ weakly (see equation (\ref{eqEmpProbMeas}) and hypothesis \ref{hypWeakConv}),
equations (\ref{eqhn}) and (\ref{eqh0}) imply that this is false.
Thus equation (\ref{eqlemcp1}) is true when $j = 0$. The result for $j \ge 1$ follows from
Cauchy estimates.

Now, since $h'$ is a non-constant analytic function on $\{w \in \C : \Im(w) > 0\}$ with
$h'(w_0) = 0$, then $h'(w) \neq 0$ for all $w \in B(w_0,\e) \setminus \{w_0\}$ and all
$\e > 0$ sufficiently small. Thus, letting $\partial B(w_0,\e)$ be the boundary of
$B(w_0,\e)$, the Bolzano-Weierstrass Theorem gives
$$
\inf_{w \in  \partial B(w_0,\e)} |h'(w)| > 0,
$$
for all $\e > 0$ sufficiently small.
%
%
It thus follows from equation (\ref{eqlemcp1}) and Rouch\'{e}'s Theorem that there exists
a function $N : \R_+ \to \N$ for which $h'$ and $h_{n,v}'$ have the same number of roots in
$B(w_0,\e)$ (counting multiplicities) for all $\e>0$ sufficiently small, $v \in V$
and $n \ge N(\e)$. Since this can be done for any $\e>0$ sufficiently small, the required
results follow from the above observation that $h_{n,v}'$ has at most one root
(counting multiplicities) in $\{ w \in \C: \Im(w) > 0 \}$.
\end{proof}

For notational purposes set $w_0^+ := w_0$, $w_0^- := \overline{w_0}$,
$w_{n,v}^+ := w_{n,v}$ and $w_{n,v}^- := \overline{w_{n,v}}$. 

\begin{rem}
\label{remmu00}
Some useful observations: Equations (\ref{eqhn}) and (\ref{eqh0}), and
Lemma \ref{lemcp}, give 
$$
\int \frac{w_{n,v}^\pm}{w_{n,v}^\pm + c - x} \mu_{n,v}[dx] = \frac{n - \an}n
\hspace{0.5cm} \mbox{and} \hspace{0.5cm}
\int \frac{w_0^\pm}{w_0^\pm + c - x} \mu[dx] = 1 - \a,
$$
for all $n$ sufficiently large and $v \in V$. Comparing real and imaginary parts,
\begin{eqnarray}
\label{eqremmu001}
\int \frac{c - x}{|w_{n,v}^\pm + c - x|^2} \mu_{n,v}[dx] \; = \; 0,
\hspace{0.5cm} & & \hspace{0.5cm}
\int \frac{c - x}{|w_0^\pm + c - x|^2} \mu[dx] \; = \; 0, \\
\label{eqremmu002}
\int \frac{|w_{n,v}^\pm|^2}{|w_{n,v}^\pm + c - x|^2} \mu_{n,v}[dx] = \frac{n - \an}n,
\hspace{0.5cm} & & \hspace{0.5cm}
\int \frac{|w_0^\pm|^2}{|w_0^\pm + c - x|^2} \mu[dx] = 1 - \a,
\end{eqnarray}
for all $n$ sufficiently large and $v \in V$.
%
%
\end{rem}

We now fix the contours $\g_n$ and $\G_n$ of equation (\ref{eqIntExpCorKer}). We define
them to pass through $w_{n,v}^\pm$ so that a saddle point asymptotic analysis
can be performed, i.e., the integral can be estimated using small sections of the
contours around $w_{n.v}^\pm$. Equation (\ref{eqIntExp}) implies that we need to choose
them so that $w \mapsto \left| e^{h_n(w)} \right|$ and
$z \mapsto \left| e^{-h_n(z)} \right|$, for all $w$ on $\g_n$ and $z$ on $\G_n$,
are both maximised at $w_{n,v}^\pm$.

Lemma \ref{lemcp} and equation (\ref{eqlemcp1}) show that $h''(w_0^+) \neq 0$ and
\begin{equation}
\label{eqhmuhnbdd}
\lim_{n \to \infty} \sup_{v \in V} \left| h_{n,v}''(w_{n,v}^\pm) - h''(w_0^\pm) \right| = 0.
\end{equation}
%
%
Thus we can define $\t_{n,v} := - \frac12 \Arg(h_{n,v}''(w_{n,v}^+))$ for all $n$ sufficiently
large and $v \in V$, where $\Arg(w) \in (-\pi,\pi]$ is the argument of $w$. Then,
fixing $\d \in (\frac13,\frac12)$, and defining $\varepsilon_0 := (-1)^{1_{C_0 \ge 0}}$ where
$C_0 \in \R$ is given in equation (\ref{eqC0}), define
\begin{equation}
\label{eqgamdef}
\begin{array}{llll}
\g_{n,1}^+(s) & := & |w_{n,v}^+ - i \varepsilon_0 \; n^{-\d} e^{ i \t_{n,v} }| \; e^{is},
& s \in [0,\frac1n), \\
\g_{n,2}^+(s) & := & |w_{n,v}^+ - i \varepsilon_0 \; n^{-\d} e^{ i \t_{n,v} }| \; e^{is},
& s \in [\frac1n,\Arg (w_{n,v}^+ - i \varepsilon_0 \; n^{-\d} e^{ i \t_{n,v} })), \\
\g_{n,3}^+(s) & := & w_{n,v}^+ + i \varepsilon_0 \; n^{-\d} e^{ i \t_{n,v} } s,
& s \in [-1,1], \\
\g_{n,4}^+(s) & := & |w_{n,v}^+ + i \varepsilon_0 \; n^{-\d} e^{ i \t_{n,v} }| \; e^{is},
& s \in (\Arg(w_{n,v}^+ + i \varepsilon_0 \; n^{-\d} e^{ i \t_{n,v} }),\pi-\frac1n], \\
\g_{n,5}^+(s) & := & |w_{n,v}^+ + i \varepsilon_0 \; n^{-\d} e^{ i \t_{n,v} }| \; e^{is},
& s \in (\pi-\frac1n,\pi].
\end{array}
\end{equation}
Take $\g_n := \sum_{j=1}^5 (\g_{n,j}^+ + \g_{n,j}^-)$, where $\g_{n,j}^-$ is the contour with
counter-clockwise orientation obtained by reflecting $\g_{n,j}^+$ through the real line. Also
define $\G : \{z \in \C : \Im(z) > 0\} \times (0,\infty) \to \{z \in \C : \Im(z) > 0\}$ by
\begin{equation}
\label{eqG}
\G(w,s) := s |w| \exp \left( i \; \arccos \left( \cos(\Arg(w)) \frac{2s\log(s)}{s^2-1} \right) \right),
\end{equation}
for all $w \in \C$ with $\Im(w) > 0$ and $s > 0$, where $\arccos : (-1,1) \to (0,\pi)$ is the principal
value of the inverse cosine function. Note, for any fixed $w \in \C$ with $\Im(w) > 0$, the contour
$\G(w,\cdot) : (0,\infty) \to \{z \in \C : \Im(z) > 0\}$ is well-defined and continuous and satisfies
$\G(w,1) = w$ and $\lim_{s \to 0} \Arg(\G(w,s)) = \lim_{s \to \infty} \Arg(\G(w,s)) = \frac{\pi}2$.
%
%
Then, for all $n$ sufficiently large, $u \in U$ and $v \in V$, define
\begin{equation}
\label{eqGamndef}
\begin{array}{llll}
\G_{n,1}^+(s) & := & (1-s) y_{n,u,v} + s \; \G(w_{n,v}^+ - \varepsilon_0 \; n^{-\d} e^{ i \t_{n,v} }, t_0),
& s \in [0,1), \\
\G_{n,2}^+(s) & := & \G(w_{n,v}^+ - \varepsilon_0 \; n^{-\d} e^{ i \t_{n,v} }, s), & s \in [t_0,1), \\
\G_{n,3}^+(s) & := & w_{n,v}^+ + \varepsilon_0 \; n^{-\d} e^{ i \t_{n,v} } s, & s \in [-1,1], \\
\G_{n,4}^+(s) & := & \G(w_{n,v}^+ + \varepsilon_0 \; n^{-\d} e^{ i \t_{n,v} }, s), & s \in (1,T_0),
\end{array}
\end{equation}
where $t_0 \in (0,1)$, $T_0>1$, and $y_{n,u,v} \in (x_j^{(n)} - \frac{v}n - c, x_{j-1}^{(n)}
- \frac{v}n - c)$ for that value of $j$ which satisfies $c + \frac{u}n \in [x_j^{(n)},x_{j-1}^{(n)})$.
These quantities will be fixed in Lemma \ref{lemBdRem}. Finally define $\G_{n,5}^+$ to be the
contour that spans the segment of the circle centered at the origin, starting at
$\G(w_{n,v}^+ + \varepsilon_0 \; n^{-\d} e^{ i \t_{n,v} }, T_0)$, ending on the real line,
with clockwise orientation when $v \le u$ and counter-clockwise orientation when $v > u$.
Take $\G_n := \sum_{j=1}^5 (\G_{n,j}^+ + \G_{n,j}^-)$, where $\G_{n,j}^-$ is  the contour
obtained by reflecting $\G_{n,j}^+$ through the real line. 

Let $\g : [0,\pi] \to \C$ be the contour given by $\g(s) := |w_0^+| e^{is}$ for all $s \in [0,\pi]$.
It follows from Lemma \ref{lemcp} and equation (\ref{eqgamdef}) that this can be regarded as
the `limit contour' in the upper half complex plane of $\{\g_n\}_{n \ge 1}$. Also equations
(\ref{eqG}) and (\ref{eqGamndef}) show that $\G(w_0^+,\cdot) : (0,\infty) \to
\{z \in \C : \Im(z) > 0\}$ can be regarded as the `limit contour' in the upper half complex plane
of $\{\G_n\}_{n \ge 1}$. As we shall see in Lemmas \ref{lemfn1fn2} and \ref{lempn2}, the functions
given by $w \mapsto \left| e^{h(w)} \right|$ and $z \mapsto \left| e^{-h(z)} \right|$,
for all $w$ on $\g$ and $z$ on $\G(w_0^+,\cdot)$, are both maximised at $w_0^+$. Using
this fact, the properties of $\g_n$ and $\G_n$ that make them suitable for saddle point analysis
are shown in Lemmas \ref{lemfn1fn2} and \ref{lempn2}. As we shall see,
for $n$ sufficiently large, the only significant contributions come from
$\g_{n,3}^\pm$ and $\G_{n,3}^\pm$.

For all $n$ sufficiently large, $u \in U$, $v \in V$ and $A,B \subset \{1,2,3,4,5\}$, define
$\g_n^A := \sum_{j \in A} \left( \g_{n,j}^+ + \g_{n,j}^- \right)$, 
$\G_n^B := \sum_{j \in B} \left( \G_{n,j}^+ + \G_{n,j}^- \right)$, and
\begin{equation}
\label{eqknsub}
K_{n,u,v}^{A,B}
:= \int_{\g_n^A} dw \int_{\G_n^B} dz
\left( \frac{(z + \frac{v - u}n)^{n-\an} - z^{n-\an}}{(v-u) w^{n-\an+1}} \right)
\sum_{j=1}^n \frac{c + \frac{v}n - x_j^{(n)}}{(z + c + \frac{v}n - x_j^{(n)})^2}
\prod_{i \neq j} \left( \frac{w + c + \frac{v}n - x_i^{(n)}}{z + c + \frac{v}n - x_i^{(n)}} \right).
\end{equation}
For all $n$ sufficiently large, $u \in U$ and $v \in V$, equation (\ref{eqIntExpCorKer}) then gives
\begin{eqnarray*}
& & \hspace{2cm}
\frac{4 \pi^2}n K_n \left( \left( \an, c + \frac{u}n \right), \left( \an, c + \frac{v}n \right) \right) \\
& = & K_{n,u,v}^{3,3} \; + \; K_{n,u,v}^{\{1,2,4,5\},3} \; + \; K_{n,u,v}^{\{1,2,3,4,5\},1}
\; + \; K_{n,u,v}^{\{1,2,3,4,5\},\{2,4\}} \; + \; K_{n,u,v}^{\{1,2,3,4,5\},5}.
\end{eqnarray*}
Lemmas \ref{lemBdRem}, \ref{lemfn1fn2} and \ref{lempn2}
imply that there exists constants $c>0$ and $C>1$ for which
\begin{eqnarray*}
\sup_{u \in U, v \in V} \left| K_{n,u,v}^{\{1,2,3,4,5\},1} + K_{n,u,v}^{\{1,2,3,4,5\},5} \right|
& \le & C^{-n}, \\
\sup_{u \in U, v \in V} \left| K_{n,u,v}^{\{1,2,4,5\},3} + K_{n,u,v}^{\{1,2,3,4,5\},\{2,4\}} \right|
& \le & n^3 \; e^{- c n^{1-2\d}},
\end{eqnarray*}
for all $n$ sufficiently large.
%
%
We now give a proof of Theorem \ref{thSine}:

\begin{proof}[\textbf{Proof of Theorem \ref{thSine}:}]
The first part of this Theorem was shown in Lemma \ref{lemcp}
(note, in Lemma \ref{lemcp} we denoted $w_{\a,c}$ by $w_0$ for simplicity of notation).
It remains to show the asymptotic limit.  Since $\d \in (\frac13,\frac12)$, the above
equation and bounds imply that the result follows if and only if
\begin{equation}
\label{eqthSine1}
\lim_{n \to \infty} \sup_{u \in U,v \in V} \left| K_{n,u,v}^{3,3}
- 4\pi (C_{\a,c})^{-\rho_\a(c) (v-u)} \frac{\sin \left( \pi \rho_\a (c) (v - u) \right)}{v-u} \right| = 0,
\end{equation}
where $\rho_{\a} (c) = - \frac{1-\a}{\pi} \Im \left( w_0^{-1} \right)$
and $C_{\a,c} = \exp \left( \pi \; \frac{ \Re (w_0^{-1}) }{ \Im (w_0^{-1}) } \right)$.
%
%

Equations (\ref{eqIntExp}) and (\ref{eqknsub}) give
\begin{equation}
\label{eqkn33}
K_{n,u,v}^{3,3} = K_{n,u,v}^{++} + K_{n,u,v}^{--} + K_{n,u,v}^{+-} + K_{n,u,v}^{-+},
\end{equation}
for all $n$ sufficiently large, $u \in U$ and $v \in V$, where for $b,d \in \{-,+\}$,
$$
K_{n,u,v}^{bd} := n \int_{\g_{n,3}^d} dw \int_{\G_{n,3}^b} dz \;
\left( \frac{( 1 + \frac{v - u}{n z})^{n-\an} - 1}{v-u} \right) g_{n,v}(w,z)
\; e^{n(h_{n,v}(w) - h_{n,v}(z))},
$$
and $h_{n,v}, g_{n,v}$ are defined in equations (\ref{eqhn}) and (\ref{eqhngn}). Also,
recalling that $h_{n,v}'(w_{n,v}^\pm) = 0$ (see Lemma \ref{lemcp}), equations (\ref{eqlemcp1}),
(\ref{eqgamdef}) and (\ref{eqGamndef}) and Taylor expansions give
\begin{equation}
\label{eqTay}
\begin{array}{rcl}
h_{n,v}(\g_{n,3}^b(s))
& = & h_{n,v} (w_{n,v}^b) - \frac12 n^{-2\d} \left| h_{n,v}''(w_{n,v}^+) \right| s^2 + R_{n,v}^b(s), \\
h_{n,v} (\G_{n,3}^d(t))
& = & h_{n,v}(w_{n,v}^d) + \frac12 n^{-2\d} \left| h_{n,v}''(w_{n,v}^+) \right| t^2 + T_{n,v}^d(t),
\end{array}
\end{equation}
for all $n$ sufficiently large, $v \in V$, $b,d \in \{-,+\}$ and $s,t \in [-1,1]$, where
the remainders satisfy
\begin{equation}
\label{eqRnbs}
\sup_{(s,t) \in [-1,1]^2} \left| R_{n,v}^b(s) \right| + \left| T_{n,v}^d(t) \right|
\; \le \; C n^{-3\d},
\end{equation}
for some constant $C>0$.
Therefore equations (\ref{eqgamdef}) and (\ref{eqGamndef}) give
$$
K_{n,u,v}^{bd} = n^{1-2\d} \; C_{n,v}^{bd} e^{n(h_{n,v}(w_{n,v}^b) - h_{n,v}(w_{n,v}^d))}
\int_{-1}^1 ds \int_{-1}^1 dt \; A_{n,u,v}^{bd}(s,t) \;
e^{- \frac12 n^{1-2\d} |h_{n,v}''(w_{n,v}^+)| (s^2 + t^2)},
$$
for all $n$ sufficiently large, $u \in U$, $v \in V$ and $b,d \in \{-,+\}$, where
\begin{eqnarray*}
C_{n,v}^{bd}
& := & \left\{ \begin{array}{ccc}
b \; i e^{- i \Arg(h_{n,v}''(w_{n,v}^b))} & ; & b=d, \\
- b \; i & ; & b \neq d,
\end{array} \right. \\
A_{n,u,v}^{bd}(s,t) \
& := & \left( \frac{(1 + \frac{v - u}{n \G_{n,3}^d(t)})^{n-\an} - 1}{v-u} \right)
g_{n,v}(\g_{n,3}^b(s), \G_{n,3}^d(t)) e^{n(R_{n,v}^b(s) - T_{n,v}^d(t))},
\end{eqnarray*}
for all $s,t \in [-1,1]$.
%
%
Then, letting
$$
A_{n,u,v}^{bd} := \left( \frac{e^{(1 - \a) (v - u) (w_0^d)^{-1} } - 1}{v-u} \right)
g_{n,v}(w_{n,v}^b, w_{n,v}^d),
$$
and noting that $|C_{n,v}^{bd}| = |e^{n(h_{n,v}(w_{n,v}^b) - h_{n,v}(w_{n,v}^d))}| = 1$,
\begin{eqnarray*}
\lefteqn{\left| K_{n,u,v}^{bd} - \frac{2\pi C_{n,v}^{bd} A_{n,u,v}^{bd}}{|h_{n,v}''(w_{n,v}^+)|}
e^{n(h_{n,v}(w_{n,v}^b) - h_{n,v}(w_{n,v}^d))} \right|} \\
& \le & n^{1-2\d} \left( \sup_{(x,y) \in [-1,1]^2} \left| A_{n,u,v}^{bd}(x,y) - A_{n,u,v}^{bd} \right| \right)
\int_{-1}^1 ds \int_{-1}^1 dt \; e^{- \frac12 n^{1-2\d} |h_{n,v}''(w_{n,v}^+)| (s^2 + t^2)}  \\
& + & \left| A_{n,u,v}^{bd} \right| \; \left| n^{1-2\d}
\int_{-1}^1 ds \int_{-1}^1 dt \; e^{- \frac12 n^{1-2\d} |h_{n,v}''(w_{n,v}^+)| (s^2 + t^2)}
- \frac{2\pi}{|h_{n,v}''(w_{n,v}^+)|} \right|,
\end{eqnarray*}
for all $n$ sufficiently large, $u \in U$, $v \in V$ and $b,d \in \{-,+\}$
(note, it follows from equation (\ref{eqhmuhnbdd}) and Lemma \ref{lemcp} that 
$h_{n,v}''(w_{n,v}^+) \neq 0$ for all $n$ sufficiently large and $v \in V$, and so these
expressions are well-defined). A change of variables and equation (\ref{eqhngn}) then gives
\begin{eqnarray*}
\lefteqn{\left| K_{n,u,v}^{bd}
- b \; 2\pi i \left( \frac{e^{(1 - \a) (v - u) (w_0^b)^{-1} } - 1}{v-u} \right) \delta_{bd} \right|
\; \le \; \frac{2\pi}{|h_{n,v}''(w_{n,v}^+)|}
\sup_{(x,y) \in [-1,1]^2} \left| A_{n,u,v}^{bd}(x,y) - A_{n,u,v}^{bd} \right|} \\
& + & \delta_{bd} \left| \frac{e^{(1 - \a) (v - u) (w_0^+)^{-1} } - 1}{v-u} \right|
\left( 2\pi - \int_{-|h_{n,v}''(w_{n,v}^+)|^{\frac12}
n^{\frac12-\d}}^{|h_{n,v}''(w_{n,v}^+)|^{\frac12} n^{\frac12-\d}} ds
\int_{-|h_{n,v}''(w_{n,v}^+)|^{\frac12}
n^{\frac12-\d}}^{|h_{n,v}''(w_{n,v}^+)|^{\frac12} n^{\frac12-\d}} dt
\; e^{-\frac12 (s^2 + t^2)} \right),
\end{eqnarray*}
for all $n$ sufficiently large, $u \in U$, $v \in V$ and $b,d \in \{-,+\}$,
where $\delta_{bd} = 1$ if $b = d$ and $\delta_{bd} = 0$ otherwise.
%
%
Recalling that $\d \in (\frac13,\frac12)$ and $h''(w_0) \neq 0$ (see Lemma \ref{lemcp}),
equations (\ref{eqhngn}), (\ref{eqlemcp1}), (\ref{eqhmuhnbdd}) and (\ref{eqRnbs}) give
$$
\lim_{n \to \infty} \sup_{u \in U, v \in V}\left| K_{n,u,v}^{bd}
- b \; 2\pi i \left( \frac{e^{(1 - \a) (v - u) (w_0^b)^{-1} } - 1}{v-u} \right) \delta_{bd} \right| = 0,
$$
for all $b,d \in \{-,+\}$.
Equation (\ref{eqthSine1}) then follows from equation (\ref{eqkn33}), as required.
%
%
%
\end{proof}

\subsection{Calculations}
\label{secC}

In this section we omit superscripts when no confusion is possible. Also we denote the range of
a contour $\g$ by $\g^\ast$. Moreover recall that $\t_{n,v} = - \frac12 \Arg(h_{n,v}''(w_{n,v}^+))$
for all $n$ sufficiently large and $v \in V$, and $\varepsilon_0 := (-1)^{1_{C_0 \ge 0}}$, where
$C_0$ is given in equation (\ref{eqC0}). Finally note equations (\ref{eqgamdef}), (\ref{eqGamndef})
and Lemma \ref{lemcp} give
\begin{equation}
\label{eqdeltan}
0
\; = \; \lim_{n \to \infty} \sup_{v \in V} \sup_{w \in \g_n^\ast} ||w| - |w_0||
\; = \; \lim_{n \to \infty} \sup_{v \in V} \sup_{w \in \g_{n,3}^\ast} |w-w_0|
\; = \; \lim_{n \to \infty} \sup_{v \in V} \sup_{z \in \G_{n,3}^\ast} |z-w_0|.
\end{equation}
%
%

\begin{lem}
\label{lemBdRem}
There exists a constant $C>1$, and a choice
of $y_{n,u,v}$, $t_0$ and $T_0$ in equation (\ref{eqGamndef}) for which
\begin{eqnarray*}
\left| K_{n,u,v}^{\{1,2,3,4,5\},1} + K_{n,u,v}^{\{1,2,3,4,5\},5}\right|
& \le & C^{-n}, \\
\left| K_{n,u,v}^{\{1,2,4,5\},3} \right|
& \le & \sup_{(w,z) \in (\g_n^{2,4})^\ast \times (\G_n^3)^\ast}
n^3 \left| e^{n(h_{n,v}(w) - h_{n,v}(z))} \right|, \\
\left| K_{n,u,v}^{\{1,2,3,4,5\},\{2,4\}} \right|
& \le & \sup_{(w,z) \in (\g_n^{2,3,4})^\ast \times (\G_n^{2,4})^\ast}
n^3 \left| e^{n(h_{n,v}(w) - h_{n,v} (z))} \right|,
\end{eqnarray*}
for all $n$ sufficiently large, $u \in U$ and $v \in V$.
\end{lem}

\begin{proof}

Consider $K_{n,u,v}^{\{1,2,3,4,5\},5}$. Equation (\ref{eqknsub}) gives
\begin{eqnarray*}
\lefteqn{K_{n,u,v}^{\{1,2,3,4,5\},5}} \\
& = & \int_{\g_n} dw \int_{\G_n^5} dz
\left( \frac{(z + \frac{v - u}n)^{n-\an} - z^{n-\an}}{(v-u) w^{n-\an+1}} \right)
\sum_{j=1}^n \frac{c + \frac{v}n - x_j}{(z + c + \frac{v}n - x_j)^2}
\prod_{i \neq j} \left( \frac{w + c + \frac{v}n - x_i}{z + c + \frac{v}n - x_i} \right),
\end{eqnarray*}
for all $n$ sufficiently large, $u \in U$ and $v \in V$. Recall that, for all $n$ sufficiently
large and $v \in V$, $\G_n^5$ spans a segment of the circle centered at the origin, with
radius $T_0 |w_{n,v} + \varepsilon_0 \; n^{-\d} e^{ i \t_{n,v} }|$, where $T_0>1$ (see equations
(\ref{eqG}) and (\ref{eqGamndef})). It thus follows from equation (\ref{eqdeltan}) that there
exists a constant $C>0$, and a function $N : (0,\infty) \to \N$ for which
\begin{equation}
\label{eqKn123455}
\left| K_{n,u,v}^{\{1,2,3,4,5\},5} \right| \le C^n T_0^{-\an},
\end{equation}
for all $T_0$ sufficiently large, $n \ge N(T_0)$, $u \in U$ and $v \in V$.

\vspace{0.5cm}

Consider $K_{n,u,v}^{\{1,2,3,4,5\},\{2,4\}}$. Equation (\ref{eqknsub}) gives
\begin{eqnarray*}
\lefteqn{K_{n,u,v}^{\{1,2,3,4,5\},\{2,4\}}} \\
& = & \int_{\g_n} dw \int_{\G_n^{2,4}} dz
\left( \frac{(z + \frac{v - u}n)^{n-\an} - z^{n-\an}}{(v-u)w^{n-\an+1}} \right)
\sum_{j=1}^n \frac{c + \frac{v}n - x_j}{(z + c + \frac{v}n - x_j)^2}
\prod_{i \neq j} \left( \frac{w + c + \frac{v}n - x_i}{z + c + \frac{v}n - x_i} \right),
\end{eqnarray*}
for all $n$ sufficiently large, $u \in U$ and $v \in V$. Recall that, for all $n$ sufficiently
large and $v \in V$, $\G_{n,2}^+(s) = \G(w_{n,v} - \varepsilon_0 \; n^{-\d} e^{ i \t_{n,v} }, s)$
for all $s \in [t_0,1)$ and $\G_{n,4}^+(s) = \G(w_{n,v} + \varepsilon_0 \; n^{-\d} e^{ i \t_{n,v} }, s)$
for all $s \in (1,T_0)$, where $\G$ is given in equation (\ref{eqG}). Equations (\ref{eqG}) and
(\ref{eqdeltan}) imply that there exists an $\e > 0$ for which
$$
\inf_{v \in V} \inf_{x \in \R, z \in (\G_n^{2,4})^\ast} |z - x| \ge t_0 \e,
$$
for all $t_0$ sufficiently small and $n$ sufficiently large (chosen independently).
%
%
Therefore we can choose the constant $C>0$, and the function $N : (0,\infty) \to \N$, so that
\begin{eqnarray*}
\lefteqn{\left| K_{n,u,v}^{\{1,2,3,4,5\},\{2,4\}} \right|} \\
& \le & C \frac{T_0^2}{t_0} \left( 1 + e^{C t_0^{-1}} \right)
\sup_{(w,z) \in \g_n^\ast \times (\G_n^{2,4})^\ast} \frac{|z|^{n - \an}}{|w|^{n - \an}}
\sum_{j=1}^n \frac1{|z + c + \frac{v}n - x_j|}
\prod_{i \neq j} \left| \frac{w + c + \frac{v}n - x_i}{z + c + \frac{v}n - x_i} \right|,
\end{eqnarray*}
for all $t_0$ sufficiently small, $n \ge N(t_0)$, $u \in U$ and $v \in V$.
%
%
Equations (\ref{eqhn}), (\ref{eqgamdef}) and (\ref{eqdeltan}) then show that
we can choose $C$ and $N$ so that
\begin{equation}
\label{eqKn1234524}
\left| K_{n,u,v}^{\{1,2,3,4,5\},\{2,4\}} \right|
\le C n^2 \; \frac{T_0^2}{t_0} \left( 1 + e^{C t_0^{-1}} \right)
\sup_{(w,z) \in (\g_n^{2,3,4})^\ast \times (\G_n^{2,4})^\ast}
\left| e^{n(h_{n,v}(w) - h_{n,v}(z))} \right|,
\end{equation}
for all $t_0$ sufficiently small, $n \ge N(t_0)$, $u \in U$ and $v \in V$.
Similarly we can choose $C$ so that
\begin{equation}
\label{eqKn12453}
\left| K_{n,u,v}^{\{1,2,4,5\},3} \right|
\le C n^2 \sup_{(w,z) \in (\g_n^{2,4})^\ast \times (\G_n^3)^\ast}
\left| e^{n(h_{n,v}(w) - h_{n,v}(z))} \right|,
\end{equation}
for all $n$ sufficiently large, $u \in U$ and $v \in V$.
%
%

\vspace{0.5cm}

Consider $K_{n,u,v}^{\{1,2,3,4,5\},1}$. Recall that
$\G_{n,1}^+(s) = (1-s) y_{n,u,v} + s \; \G(w_{n,v} - \varepsilon_0 \; n^{-\d} e^{ i \t_{n,v} }, t_0)$
for all $s \in [0,1)$, where $y_{n,u,v} \in (x_j - \frac{v}n - c, x_{j-1} - \frac{v}n - c)$ for
that value of $j$ which satisfies $c + \frac{u}n \in [x_j,x_{j-1})$ (see equation (\ref{eqGamndef})).
Also recall that $|G(w,s)| = s|w|$ for all $s>0$ and $w \in \C$ with $\Im(w) > 0$. It thus
follows from equation (\ref{eqdeltan}) that we can choose $y_{n,u,v}$ so that
\begin{equation}
\label{eqGn1bdd}
\sup_{u \in U, v \in V} \sup_{z \in (\G_n^1)^\ast} |z| \le 2 t_0 |w_0|,
\end{equation}
for all $n$ sufficiently large.
%
%
Thus there exists a choice of $N : (0,\infty) \to \N$ for
which the following is well-defined for all $t_0$ sufficiently small, $n \ge N(t_0)$,
$u \in U$ and $v \in V$:
$$
\int_{\g_n} dw \int_{\G_n^1} dz \;
\frac{(z + \frac{v - u}n)^{n - \an - 1} }{w^{n - \an + 1}} \frac1{w-z}
\prod_{i=1}^n \left( \frac{w +c - \frac{v}n - x_i}{z + c - \frac{v}n - x_i} \right).
$$
Using Cauchy's Theorem to perturb the contours in a similar manner to that described
in Proposition \ref{prIntExpKl}, this quantity can be related to $K_{n,u,v}^{\{1,2,3,4,5\},1}$
in the following way:
\begin{eqnarray*}
\lefteqn{K_{n,u,v}^{\{1,2,3,4,5\},1}
= \left( \frac{n - \an}n \right) \int_{\g_n} dw \int_{\G_n^1} dz \;
\frac{(z + \frac{v - u}n)^{n - \an - 1} }{w^{n - \an + 1}} \frac1{w-z}
\prod_{i=1}^n \left( \frac{w + c - \frac{v}n - x_i}{z + c - \frac{v}n - x_i} \right)} \\
& - & \int_{\g_n} dw \;
\left( \frac{(\G_{n,1}^-(0) + \frac{v - u}n)^{n-\an} - \G_{n,1}^-(0)^{n-\an}}{(v-u)w^{n - \an + 1}} \right)
\frac{\G_{n,1}^-(0)}{w - \G_{n,1}^-(0)}
\prod_{i=1}^n \left( \frac{w + c - \frac{v}n - x_i}{\G_{n,1}^-(0) + c - \frac{v}n - x_i} \right) \\
& + & \int_{\g_n} dw \;
\left( \frac{(\G_{n,1}^+(1) + \frac{v - u}n)^{n-\an} - \G_{n,1}^+(1)^{n-\an}}{(v-u)w^{n - \an + 1}} \right)
\frac{\G_{n,1}^+(1)}{w - \G_{n,1}^+(1)}
\prod_{i=1}^n \left( \frac{w + c - \frac{v}n - x_i}{\G_{n,1}^+(1) + c - \frac{v}n - x_i} \right),
\end{eqnarray*}
for all $t_0$ sufficiently small, $n \ge N(t_0)$, $u \in U$ and $v \in V$.
Therefore we can choose $C$ and $N$ so that
$$
\left| K_{n,u,v}^{\{1,2,3,4,5\},1} \right|
\le C^n \sup_{z \in (\G_n^1)^\ast} \left| z + \frac{v - u}n \right|^{n - \an - 1}
\prod_{i=1}^n \frac1{|z + c - \frac{v}n - x_i|},
$$
for all $t_0$ sufficiently small, $n \ge N(t_0)$, $u \in U$ and $v \in V$.
It thus follows from equation (\ref{eqGn1bdd}) that, for any fixed $\e > 0$,
we can choose $C$ and $N$ so that
$$
\left| K_{n,u,v}^{\{1,2,3,4,5\},1} \right|
\le C^n \sup_{z \in (\G_n^1)^\ast} \left| z + \frac{v - u}n \right|^{n - \an - 1}
\prod_{i ; |x_i - c| < \e} \frac1{|z + c - \frac{v}n - x_i|},
$$
for all $t_0$ sufficiently small, $n \ge N(t_0)$, $u \in U$ and $v \in V$.
%
%
Recalling that $c \in A_\a$, and so $\mu[\{c\}] < 1 - \a$ (see the proof of Proposition \ref{prAa}),
we fix $\e > 0$ so that $\frac1n \# \{i : |x_i - c| < \e \} < \frac12 ( 1 - \a + \mu[\{c\}])$
for all $n$ sufficiently large (which is always possible by hypothesis \ref{hypWeakConv}).
%
%
Write
$$
\left| K_{n,u,v}^{\{1,2,3,4,5\},1} \right|
\le C^n \sup_{z \in (\G_n^1)^\ast}
\left| z + \frac{v - u}n \right|^{n - \an - 1 - \# \{i : |x_i - c| < \e \}}
\prod_{i ; |x_i - c| < \e} \left| \frac{z + \frac{v - u}n}{z + c - \frac{v}n - x_i} \right|,
$$
for all $t_0$ sufficiently small, $n \ge N(t_0)$, $u \in U$ and $v \in V$.
Equation (\ref{eqGn1bdd}) then shows that we can choose $C$ and $N$ so that
$$
\left| K_{n,u,v}^{\{1,2,3,4,5\},1} \right|
\le C^n \; (t_0)^{\frac{n}2 ( 1 - \a - \mu[\{c\}] )} \;
\sup_{z \in (\G_n^1)^\ast} \prod_{i ; |x_i - c| < \e}
\left( 1 + \left| \frac{\frac{v - u}n - (c - \frac{v}n - x_i)}{z + c - \frac{v}n - x_i} \right| \right),
$$
for all $t_0$ sufficiently small, $n \ge N(t_0)$, $u \in U$ and $v \in V$. Moreover, for any fixed
$\xi > 0$, equations (\ref{eqG}) and (\ref{eqGamndef}) imply that we can choose $N$ so that
$\sup_{v \in V} |\Arg(\G_{n,1}^+(1)) - \frac{\pi}2| \le \xi$ for all $t_0$ sufficiently small and
$n \ge N(t_0)$.
%
%
It thus follows that we can choose $N$ so that
\begin{eqnarray*}
\left| K_{n,u,v}^{\{1,2,3,4,5\},1} \right|
& \le & C^n \; (t_0)^{\frac{n}2 ( 1 - \a - \mu[\{c\}] )} \;
\prod_{i ; |x_i - c| < \e}
\left( 1 + 2 \left| \frac{\frac{v - u}n - (c - \frac{v}n - x_i)}
{y_{n,u,v} + c - \frac{v}n - x_i} \right| \right) \\
& \le & C^n \; (t_0)^{\frac{n}2 ( 1 - \a - \mu[\{c\}] )} \;
\prod_{i ; |x_i - c| < \e}
\left( 1 + 2 \left( 1 + \left| \frac{y_{n,u,v} + \frac{v - u}n}
{y_{n,u,v} + c - \frac{v}n - x_i} \right| \right) \right),
\end{eqnarray*}
for all $t_0$ sufficiently small, $n \ge N(t_0)$, $u \in U$ and $v \in V$.
%
%
%
Finally, choosing $y_{n,u,v}$ sufficiently close to $\frac{u-v}n$
(recall that $y_{n,u,v} \in (x_j - \frac{v}n - c, x_{j-1} - \frac{v}n - c)$ for that
value of $j$ which satisfies $c + \frac{u}n \in [x_j,x_{j-1})$),
\begin{equation}
\label{eqKn123451}
\left| K_{n,u,v}^{\{1,2,3,4,5\},1} \right|
\le (5 C)^n (t_0)^{\frac{n}2 ( 1 - \a - \mu[\{c\}] )},
\end{equation}
for all $t_0$ sufficiently small, $n \ge N(t_0)$, $u \in U$ and $v \in V$.
The required result follows from equations (\ref{eqKn123455}),
(\ref{eqKn1234524}), (\ref{eqKn12453}) and (\ref{eqKn123451})
by fixing $t_0$ sufficiently small and $T_0$ sufficiently large.
\end{proof}

\begin{lem}
\label{lemfn1fn2}
For $n$ all sufficiently large and $v \in V$,
$$
\sup_{w \in (\g_n^3)^\ast} \Re ( h_{n,v}(w) - h_{n,v}(w_{n,v})) = 0.
$$
Moreover, letting $\d \in (\frac13,\frac12)$ be that used in equations (\ref{eqgamdef}) and
(\ref{eqGamndef}), there exists a constant $c>0$ for which
$$
\lim_{n \to \infty} \sup_{v \in V} \sup_{w \in (\g_n^{2,4})^\ast}
n^{2\d} \Re ( h_{n,v}(w) - h_{n,v}(w_{n,v})) \le - c.
$$
\end{lem}

\begin{proof}
Consider $\g_n^3$. For all $n$ sufficiently large and $v \in V$ define
$f_{n,v} : [-1,1] \to \R$ by $f_{n,v}(s) := \Re ( h_{n,v}(\g_{n,3}^+(s)) )$ for all $s \in [-1,1]$.
Equation (\ref{eqgamdef}) gives $f_{n,v}'(0) = 0$ for all $n$ sufficiently large and $v \in V$.
Also Lemma \ref{lemcp} and equation (\ref{eqlemcp1}) give $h''(w_0) \neq 0$ and
$$
\sup_{s \in [-1,1]} \sup_{v \in V} f_{n,v}''(s)
\le - n^{-2\d} \sup_{v \in V} \sup_{w \in B(w_{n,v},n^{-\d})}
\Re \left( h_{n,v}''(w) e^{- i \Arg(h_{n,v}''(w_{n,v}))} \right)
\le - \frac12 n^{-2\d} | h''(w_0) |,
$$
for all $n$ sufficiently large, as required.
%
%

Now consider $\g_n^{2,4}$. Define $z_{n,v}^\pm := w_{n,v} \pm i n^{-\d} e^{i \t_{n,v}}$
for all $n$ sufficiently large and $v \in V$, where $\t_{n,v} = -\frac12 \Arg(h_{n,v}''(w_{n,v}))$.
Also define $f, f_{n,v}^\pm : (0,\pi) \to \R$ by
\begin{equation}
\label{eqf0fnpm}
f(s) := \Re \left( h \left( |w_0| e^{is} \right) \right)
\hspace{1cm} \mbox{and} \hspace{1cm}
f_{n,v}^\pm(s) := \Re \left( h_{n,v} \left( |z_{n,v}^\pm| e^{is} \right) \right),
\end{equation}
for all $n$ sufficiently large, $v \in V$ and $s \in (0,\pi)$.
Equations (\ref{eqhn}) and (\ref{eqh0}) give
\begin{eqnarray*}
f(s) & = & \frac12 \int \log \left| |w_0| e^{is} + c - x \right|^2 \mu[dx] - (1 - \a) \log |w_0|, \\
f_{n,v}^\pm(s) & = & \frac12 \int \log \left| |z_{n,v}^\pm| e^{is} + c - x \right|^2 \mu_{n,v}[dx]
- \frac{n - \an}n \log |z_{n,v}^\pm|,
\end{eqnarray*}
for all $n$ sufficiently large, $v \in V$ and $s \in (0,\pi)$. It is easy to see that $f''(s) < 0$
for any $s \in (0,\pi)$ with $f'(s) = 0$. Also equation (\ref{eqremmu001}) gives $f'(\Arg(w_0)) = 0$.
%
%
Thus $f$ has a unique critical point in $(0,\pi)$, a global maximum at $\Arg(w_0)$.
Similarly, for all $n$ sufficiently large and $v \in V$, $f_{n,v}^\pm$ has at most
one critical point in $(0,\pi)$ which, if it exists, must be a global maximum.
%
%
To demonstrate it's existence, fix $\phi_0 \in (0,\frac{\pi}2)$. Since $\frac{q_n}n \to \a$
and $\mu_{n,0} \to \mu$ weakly as $n \to \infty$ (see equation (\ref{eqEmpProbMeas}) and
hypothesis \ref{hypWeakConv}), Lemma \ref{lemcp} implies that
$$
\lim_{n \to \infty} \sup_{v \in V} \sup_{s \in [\phi_0,\pi-\phi_0]} |f_{n,v}^\pm(s) - f(s)| = 0.
$$
Thus, for all $n$ sufficiently large and $v \in V$, since the unique critical point of $f$ in
$(0,\pi)$ is a global maximum at $\Arg(w_0)$, $f_{n,v}^\pm$ must have a unique critical point
in $(0,\pi)$. Denoting by $s_{n,v}^\pm$, it also follows that
$\sup_{v \in V} |s_{n,v}^\pm - \Arg(w_0)| \to 0$ as $n \to \infty$.
%
%

Equation (\ref{eqremmu001}) gives
$$
(f_{n,v}^\pm)'(\Arg(z_{n,v}^\pm)) = \Im (z_{n,v}^\pm) \int \left( \frac{c - x}{|w_{n,v} + c - x|^2}
- \frac{c - x}{|z_{n,v}^\pm + c - x|^2} \right) \mu_{n,v}[dx],
$$
for all $n$ sufficiently large and $v \in V$.
%
%
Then, since $\mu_{n,0} \to \mu$ weakly as $n \to \infty$ (see equation (\ref{eqEmpProbMeas})
and hypothesis \ref{hypWeakConv}),
\begin{equation}
\label{eqfnpmdArg}
\lim_{n \to \infty} \sup_{v \in V} \left| n^\d (f_{n,v}^\pm)'(\Arg(z_{n,v}^\pm)) \mp  C_0 \right| = 0,
\end{equation}
where
\begin{equation}
\label{eqC0}
C_0 := 2 \Im(w_0) \int \frac{(c - x) \Im( (w_0 + c - x) e^{\frac{i}2 \Arg(h''(w_0))})}
{|w_0 + c - x|^4} \mu[dx].
\end{equation}
Also equation (\ref{eqremmu001}) gives
$$
(f_{n,v}^\pm)''(s) = \int \left( \frac{(c - x) |z_{n,v}^\pm| \cos(s)}{|w_{n,v} + c - x|^2} -
\frac{(c - x) |z_{n,v}^\pm| \cos(s)}{| |z_{n,v}^\pm| e^{is} + c - x |^2}
- 2 \left( \frac{(c - x) |z_{n,v}^\pm| \sin (s)}{| |z_{n,v}^\pm| e^{is} + c - x |^2} \right)^2 \right) \mu_{n,v}[dx],
$$
for all $n$ sufficiently large, $v \in V$ and $s \in (0,\pi)$. Thus there exists a constant $C>0$
for which
$$
\left| (f_{n,v}^\pm)''(s) + 2 \Im(w_0)^2 \int \frac{(c - x)^2}{|w_0 + c - x|^4} \mu_{n,v}[dx] \right|
\le C (|s - \Arg(w_0)| + |w_{n,v} - w_0| + n^{-\d}),
$$
for all $n$ sufficiently large, $v \in V$ and $s$ sufficiently close to $\Arg(w_0)$.
%
%
Thus, since $\mu_{n,0} \to \mu$ weakly as $n \to \infty$ (see equation (\ref{eqEmpProbMeas})
and hypothesis \ref{hypWeakConv}), Lemma \ref{lemcp} implies that there exists an $\e>0$ for
which $(f_{n,v}^\pm)''(s) < - \e$ for all $n$ sufficiently large, $v \in V$ and $s$
sufficiently close to $\Arg(w_0)$. 

Consider the case $C_0>0$. Then $\varepsilon_0 = -1$ (recall
$\varepsilon_0 = (-1)^{1_{C_0 \ge 0}}$) and, for all $n$ sufficiently large and $v \in V$,
equations (\ref{eqgamdef}) and (\ref{eqf0fnpm}) give
\begin{eqnarray*}
\sup_{w \in (\g_n^2)^\ast} \Re ( h_{n,v}(w) - h_{n,v}(\g_{n,3}^+(-1)) )
& = & \sup_{s \in [\frac1n,\Arg(z_{n,v}^+)]} ( f_{n,v}^+(s) - f_{n,v}^+(\Arg(z_{n,v}^+)) ), \\
\sup_{w \in (\g_n^4)^\ast}
\Re ( h_{n,v}(w) - h_{n,v}(\g_{n,3}^+(1)) )
& = & \sup_{s \in [\Arg(z_{n,v}^-),\pi-\frac1n]} ( f_{n,v}^-(s) - f_{n,v}^-(\Arg(z_{n,v}^-)) ).
\end{eqnarray*}
Also equation (\ref{eqfnpmdArg}) gives $(f_{n,v}^-)'(\Arg(z_{n,v}^-)) < 0 < (f_{n,v}^+)'(\Arg(z_{n,v}^+))$
for all $n$ sufficiently large and $v \in V$. Recall that the unique critical point of $f_{n,v}^\pm$ in
$(0,\pi)$ is a global maximum at $s_{n,v}^\pm$. Thus for all $n$ sufficiently large and $v \in V$,
$\Arg(z_{n,v}^+) < s_{n,v}^+$, $\Arg(z_{n,v}^-) > s_{n,v}^-$, and
$$
\sup_{w \in (\g_n^2)^\ast} \Re ( h_{n,v}(w) - h_{n,v}(\g_{n,3}^+(-1)) )
= 0 =
\sup_{w \in (\g_n^4)^\ast} \Re ( h_{n,v}(w) - h_{n,v}(\g_{n,3}^+(1)) ).
$$
It thus follows that
$$
\sup_{w \in (\g_n^{2,4})^\ast} \Re ( h_{n,v}(w) - h_{n,v}(w_{n,v}) )
= \Re ( h_{n,v}(\g_{n,3}^+(\pm 1)) - h_{n,v}(w_{n,v}) ),
$$
for all $n$ sufficiently large and $v \in V$. Thus, since $\d \in (\frac13,\frac12)$
and $h''(w_0) \neq 0$ (see Lemma \ref{lemcp}), equations (\ref{eqhmuhnbdd}), (\ref{eqTay})
and (\ref{eqRnbs}) give the required result.
%
%
Similarly for $C_0 < 0$.

Now suppose $C_0=0$. Recall that there exists an
$\e>0$ for which $(f_{n,v}^\pm)''(s) < - \e$ for all $n$ sufficiently large, $v \in V$ and $s$
sufficiently close to $\Arg(w_0)$. Thus, since $\sup_{v \in V} |s_{n,v}^\pm - \Arg(w_0)| \to 0$
and $\sup_{v \in V} |z_{n,v}^\pm - w_0| \to 0$ as $n \to \infty$,
$$
\e |s_{n,v}^\pm - \Arg(z_{n,v}^\pm)| \le |(f_{n,v}^\pm)'(\Arg(z_{n,v}^\pm))|,
$$
for all $n$ sufficiently large and $v \in V$.
%
%
Moreover
$$
f_{n,v}^\pm(s_{n,v}^\pm) - f_{n,v}^\pm(\Arg(z_{n,v}^\pm))
\le |s_{n,v}^\pm - \Arg(z_{n,v}^\pm)| \; \left| (f_{n,v}^\pm)'(\Arg(z_{n,v}^\pm)) \right|,
$$
for all $n$ sufficiently large and $v \in V$,
%
%
and so
$$
f_{n,v}^\pm(s_{n,v}^\pm) - f_{n,v}^\pm(\Arg(z_{n,v}^\pm))
\le \e^{-1} \left| (f_{n,v}^\pm)'(\Arg(z_{n,v}^\pm)) \right|^2.
$$
Thus, since $C_0=0$, equation (\ref{eqfnpmdArg}) gives
$$
\lim_{n \to \infty} \sup_{v \in V} n^{2\d} \left( f_{n,v}^\pm(s_{n,v}^\pm)
- f_{n,v}^\pm(\Arg(z_{n,v}^\pm)) \right) = 0.
$$
Then, recalling that $s_{n,v}^\pm$ is the global maximum of $f_{n,v}^\pm$, equation
(\ref{eqf0fnpm}) gives
$$
\lim_{n \to \infty} \sup_{v \in V} \sup_{s \in (0,\pi)} n^{2\d}
\Re \left( h_{n,v} \left( |z_{n,v}^\pm| e^{is} \right) - h_{n,v}(z_{n,v}^\pm) \right) = 0.
$$
%
%
The required result follows from equations (\ref{eqgamdef}) and (\ref{eqTay}) in a similar way to before. 
%
%
\end{proof}

\begin{lem}
\label{lempn2}
For all $n$ sufficiently large and $v \in V$,
$$
\sup_{z \in (\G_n^3)^\ast} \Re ( h_{n,v}(w_{n,v}) - h_{n,v}(z) ) = 0.
$$
Moreover, letting $\d \in (\frac13,\frac12)$ be that used in equations (\ref{eqgamdef}) and
(\ref{eqGamndef}), there exists a constant $c>0$ for which
$$
\lim_{n \to \infty} \sup_{v \in V} \sup_{z \in (\G_n^{2,4})^\ast}
n^{2\d} \Re ( h_{n,v}(w_{n,v}) - h_{n,v}(z)) \le - c.
$$
\end{lem}

\begin{proof}
The result for $\G_n^3$ follows in a similar way to that for $\g_n^3$ given in Lemma
\ref{lemfn1fn2}. Consider $\G_n^{2,4}$. For all $n$ sufficiently large and $v \in V$,
define $z_{n,v}^\pm := w_{n,v} \pm n^{-\d} e^{i \t_{n,v}}$, and
$p, p_{n,v}^\pm :(0,\infty) \to \R$ by
$$
p(s) := -\Re(h(\G(w_0,s))
\hspace{1cm} \mbox{and} \hspace{1cm}
p_{n,v}^\pm(s) := -\Re(h_{n,v}(\G_{n,v}^\pm(z_{n,v}^\pm,s)),
$$
for all $s>0$, where $\G$ is defined in equation (\ref{eqG}).
Equations (\ref{eqhn}) and (\ref{eqh0}) give
\begin{eqnarray*}
p(s)
& = & - \frac12 \int \log |\G(w_0,s) + c - x|^2 \mu[dx] + (1 - \a) \log | \G(w_0,s) |, \\
p_{n,v}^\pm(s)
& = & - \frac12 \int \log |\G(z_{n,v}^\pm,s) + c - x|^2 \mu_{n,v}[dx]
+ \frac{n - \an}n \log | \G(z_{n,v}^\pm,s) |,
\end{eqnarray*}
for all $n$ sufficiently large, $v \in V$ and $s>0$.

Consider $p$. Recalling that $|\G(w_0,s)| = s |w_0|$ for all $s>0$ (see equation (\ref{eqG})),
equations (\ref{eqremmu001}) and (\ref{eqremmu002}) give
$$
p'(s)
= \int \left( - \frac{s |w_0|^2 + (c - x) \frac{d}{ds} \Re(\G(w_0,s))}{|\G(w_0,s) + c - x|^2}
+ \frac{|w_0|^2 + (c - x) g(s)}{s |w_0 + c - x|^2} \right) \mu[dx],
$$
for all $s > 0$, where $g : (0,\infty) \to \R$ is arbitrary.
%
%
Taking $g(s) := s \frac{d}{ds} \Re(\G(w_0,s))$ for all $s>0$, equation (\ref{eqG}) gives
$$
p'(s) = \int \frac{ q(s) (c - x)^2 }
{s |\G(w_0,s) + c - x|^2 |w_0 + c - x|^2} \mu[dx],
$$
for all $s>0$, where
$$
q(s) := (1 - s^2) \Im(w_0)^2 + \frac1{1-s^2}
\left( 1 + s^2 + \frac{4 s^2}{1-s^2} \log(s) \right)^2 \Re(w_0)^2.
$$
Note that $q(s)>0$ for all $s \in (0,1)$, $q(1)=0$, $q(s)<0$ for all $s>1$, and
$q'(1) = -2 \Im(w_0)^2$. Therefore $p$ is strictly increasing in $(0,1)$, strictly
decreasing in $(1,\infty)$, and has a global maximum at $1$.
%
%

Consider $p_{n,v}^\pm$. Proceeding in a similar way to that given above,
\begin{equation}
\label{eqpnpmd}
(p_{n,v}^\pm)'(s) = \int \frac{ r_{n,v}^\pm(s) (c - x) + q_{n,v}^\pm(s) (c - x)^2 }
{s |\G(z_{n,v}^\pm,s) + c - x|^2 |w_{n,v} + c - x|^2} \mu_{n,v}[dx],
\end{equation}
for all $n$ sufficiently large, $v \in V$ and $s>0$, where
\begin{eqnarray*}
r_{n,v}^\pm(s) & := & 2 s^2 |z_{n,v}^\pm|^2 \Re(z_{n,v}^\pm - w_{n,v}) + \frac{2 s^2}{s^2-1}
\left( 1 - \frac{2 s^2 \log(s)}{s^2-1} \right) \Re(z_{n,v}^\pm) (|z_{n,v}^\pm|^2 - |w_{n,v}|), \\
q_{n,v}^\pm(s) & := & (1 - s^2) \Im(z_{n,v}^\pm)^2 + \frac1{1-s^2}
\left( 1 + s^2 + \frac{4 s^2}{1-s^2} \log(s) \right)^2 \Re(z_{n,v}^\pm)^2 \\
& + & |w_{n,v}|^2 - |z_{n,v}^\pm|^2 + \frac{4 s^2}{(s^2-1)^2} (2 \log(s) - s^2 + 1)
\Re(z_{n,v}^\pm) \Re(w_{n,v} - z_{n,v}^\pm).
\end{eqnarray*}
These are well-defined and continuous with
\begin{equation}
\label{eqrnvqnv}
\begin{array}{rcl}
r_{n,v}^\pm(1) & = & 2 |z_{n,v}^\pm|^2 \Re(z_{n,v}^\pm - w_{n,v}) - \Re(z_{n,v}^\pm) (|z_{n,v}^\pm|^2 - |w_{n,v}|^2), \\
q_{n,v}^\pm(1) & = & |w_{n,v}|^2 - |z_{n,v}^\pm|^2 + 2 \Re(z_{n,v}^\pm) \Re(z_{n,v}^\pm - w_{n,v}),
\end{array}
\end{equation}
for all $n$ sufficiently large and $v \in V$.
Also there exists a constant $C>0$ for which
$$
\left| (p_{n,v}^\pm)''(s) + 2 \Im(w_0)^2 \int \frac{(c - x)^2}{|w_0 + c - x|^4} \mu_{n,v}[dx] \right|
\le C \left( |s - 1| + |w_{n,v} - w_0| + n^{-\d} \right),
$$
for all $n$ sufficiently large, $v \in V$ and $s$ sufficiently close to $1$.
Thus, since $\mu_{n,0} \to \mu$ weakly as $n \to \infty$ (see equation (\ref{eqEmpProbMeas})
and hypothesis \ref{hypWeakConv}), Lemma \ref{lemcp} implies that there exists an $\e>0$ for
which $(p_{n,v}^\pm)''(s) < - \e$ for all $n$ sufficiently large, $v \in V$ and $s$ sufficiently
close to $1$. Also note, since $\frac{q_n}n \to \a$ and $\mu_{n,0} \to \mu$ weakly as
$n \to \infty$, Lemma \ref{lemcp} gives
$$
\lim_{n \to \infty} \sup_{v \in V} \sup_{s \in [t_0,T_0]} |p_{n,v}^\pm(s) - p(s)| = 0.
$$
Thus for all $n$ sufficiently large and $v \in V$, since the unique critical point of $p$ in
$(0,\infty)$ is a global maximum at $1$, and since $t_0 < 1 < T_0$, $p_n^\pm$ has a unique
critical point in $[t_0,T_0]$ and this point is a local maximum.
Denoting by $s_{n,v}^\pm$, it also follows that $\sup_{v \in V} |s_{n,v}^\pm - 1| \to 0$ as
$n \to \infty$.
%
%
%

Since $\mu_{n,0} \to \mu$ weakly as $n \to \infty$, equations
(\ref{eqpnpmd}) and (\ref{eqrnvqnv}) give
$$
\lim_{n \to \infty} \sup_{v \in V} \left| n^\d (p_{n,v}^\pm)'(1) \mp  C_0 \right| = 0,
$$
where $C_0 \in \R$ is defined in equation (\ref{eqC0}).
The required result then follows from a similar argument to that given in Lemma \ref{lemfn1fn2}.
\end{proof}

\section{Unitary invariant ensemble}
\label{secUIE}

In this section we consider measures on $\overline{\GT_n}$ induced by the eigenvalue
minor process of a Unitary invariant ensemble (UIE). As in the introduction, for each
$n \in \N$, let $\H_n \subset \C^{n \times n}$ be the set of $n \times n$ complex
Hermitian matrices. Let $A_n \in \H_n$ be a random Unitary matrix with eigenvalue distribution   
$$
d\mu_n [y] = \frac1{Z_n} \Delta_n(y)^2 \left( \prod_{i=1}^n e^{- V(y_i)} \right) \; dy,
$$
for all $y \in \D_n$, where $Z_n > 0$ is a normalisation constant, $V : \R \to \R$ is a
continuous function, and $dy$ is Lebesgue measure on $\R^n$. Then $A_n$ is called a UIE
with {\em potential} $V$. For more information on UIEs see Anderson, Guionnet and Zeitouni,
\cite{And10}, and Mehta, \cite{Mehta04}. It follows from equation (\ref{eqfnGUE}) that the
GUE is the UIE with potential $V(x) = \frac12 x^2$ for all $x \in \R$.

Equations (\ref{eqEigMinProBary}) and (\ref{eqVanDet}) imply that the eigenvalue minor process
of $A_n$ has distribution
$$
d \nu_n [y^{(1)},\ldots,y^{(n)}] = \frac1{Z_n'} \Delta_n(y^{(n)})
\left( \prod_{i=1}^n e^{- V(y_i^{(n)})} \right) \; dy^{(n)} dy^{(n-1)} \ldots dy^{(1)},
$$
for all $(y^{(1)},\ldots,y^{(n)}) \in \overline{\GT_n}$, where $Z_n' > 0$ is a normalisation
constant, and $d y^{(r)}$ is Lebesgue measure on $\R^r$ for each $r$. We now use Theorem
\ref{thmGTDet} to show that $(\overline{\GT_n},\nu_n)$ is a determinantal random point
field, and obtain an expression for the correlation kernel in terms of polynomials which are
orthogonal with respect to the weight $e^{-V(\cdot)} : \R \to \R_+$. We specialise to
classical ensembles in Section \ref{secTCE}.  

\begin{prop}
\label{prUIECorKer}
For each $i,j \ge 0$, let $\psi_i, \psi_j$ be the monic polynomials of degree
$i$ and $j$ (respectively) which satisfy
\begin{equation}
\label{eqOrtPol}
\int_{-\infty}^\infty dx \; \psi_i(x) \psi_j(x) e^{-V(x)} = c_i c_j \delta_{ij},
\end{equation}
for some $c_i,c_j \in \R$. Then the random point field $(\overline{\GT_n},\nu_n)$
is determinantal with correlation kernel $K_n : (\{1,\ldots,n\} \times \R)^2 \to \C$ given by
\begin{eqnarray*}
K_n ((r,u),(s,v))
& = & 1_{v \le u} \sum_{j=n-s}^{n-1} c_j^{-2} \psi_j^{(n-s)}(v) \int_u^\infty dx \;
\frac{(x - u)^{n-r-1}}{(n-r-1)!} \psi_j (x) e^{-V(x)} \\
& - & 1_{v > u} \sum_{j=n-s}^{n-1} c_j^{-2} \psi_j^{(n-s)}(v) \int_{-\infty}^u dx \;
\frac{(x - u)^{n-r-1}}{(n-r-1)!} \psi_j (x) e^{-V(x)},
\end{eqnarray*}
for all $r \in \{1,\ldots,n-1\}$, $s \in \{1,\ldots,n\}$ and $u,v \in \R$, where $\psi_j^{(i)}$
is the $i^{\text{th}}$ derivative of $\psi_j$ for each $i,j$.
\end{prop}

\begin{proof}
It follows from equation (\ref{eqVanDet}) that
$$
d \nu_n [y^{(1)},\ldots,y^{(n)}] = \frac1{Z_n'}
\det \left[ \psi_{n-l} \left( y^{(n)}_m \right) e^{- V(y^{(n)}_m)} \right]_{l,m=1}^n
\; dy^{(n)} \ldots dy^{(1)},
$$
for all $(y^{(1)},\ldots,y^{(n)}) \in \overline{\GT_n}$. This is written in the form of
equation (\ref{eqMeasCnCn0}) with $\phi_i(x) = \psi_{n-i} (x) e^{-V(x)}$
for all $i \in \{1,\ldots,n\}$ and $x \in \R$.
The fact that $(\overline{\GT_n},\nu_n)$ is determinantal follows immediately from
Theorem \ref{thmGTDet}. The correlation kernel $K_n : (\{1,\ldots,n\} \times \R)^2 \to \C$
is given by
\begin{eqnarray*}
\lefteqn{K_n ((r,u),(s,v))} \\
& = & \frac1{B_n} \frac{\partial^{n-s}}{\partial v^{n-s}} \sum_{j=1}^n
\int_{\R^n} dy \; \left( \prod_{k=1}^n \psi_{n-k} (y_k) e^{-V(y_k)} \right)
1_{v \le u < y_j} \frac{(y_j - u)^{n-r-1}}{(n-r-1)!} \Delta_n(y_{j,v}) \\
& - & \frac1{B_n} \frac{\partial^{n-s}}{\partial v^{n-s}} \sum_{j=1}^n
\int_{\R^n} dy \; \left( \prod_{k=1}^n \psi_{n-k} (y_k) e^{-V(y_k)} \right)
1_{v > u \ge y_j} \frac{(y_j - u)^{n-r-1}}{(n-r-1)!} \Delta_n(y_{j,v}),
\end{eqnarray*}
for all $r \in \{1,\ldots,n-1\}$, $s \in \{1,\ldots,n\}$ and $u,v \in \R$,
where $y_{j,v} := (y_1,\ldots,y_{j-1},v,y_{j+1},\ldots,y_n)$ and
$$
B_n = \int_{\R^n} dy \; \left( \prod_{k=1}^n \psi_{n-k} (y_k) e^{-V(y_k)} \right) \Delta_n(y).
$$
%
%
Then, writing $\Delta_n(y) = \det[ \psi_{n-l} (y_m) ]_{l,m=1}^n$ (see equation (\ref{eqVanDet})),
equation (\ref{eqOrtPol}) gives $B_n = \prod_{k=0}^{n-1} c_k^2$.
Similarly, writing $\Delta_n(y_{j,v}) = \det[ \psi_{n-l} (y_m) 1_{m \neq j} +
\psi_{n-l}(v) 1_{m=j} ]_{l,m=1}^n$, equation (\ref{eqOrtPol}) gives
the required result.
\end{proof}

Alternatively the correlation kernel in the previous Proposition can be written as a contour
integral. Using the choice of $\phi_i : \R \to \R$ given in the previous Proposition, it follows
from Proposition \ref{prIntExpKl} that 
\begin{eqnarray*}
K_n ((r,u),(s,v))
& = & \frac1{(2\pi)^2} \frac{(n-s)!}{(n-r-1)!} \int_{\g_{u,v}} dw \int_{\G_{u,v}} dz
\; \frac{(z - u)^{n-r-1}}{(w - v)^{n-s+1}} \frac1{w-z} \times \\
& \times & \frac1{B_n} \int_{\R^n} dy \; \left( \prod_{k=1}^n \psi_{n-k} (y_k) e^{-V(y_k)} \right) \Delta_n(y)
\prod_{i=1}^n \left( \frac{w - y_i}{z - y_i} \right),
\end{eqnarray*}
for all $r \in \{1,\ldots,n-2\}$, $s \in \{1,\ldots,n\}$, and $u,v \in \R$ with $u \neq v$.
Here $\g_{u,v}$ is a counter-clockwise simple closed contour around $v$. Also $\G_{u,v} :
\R \to \C$ is a piecewise smooth contour which satisfies $\G_{u,v}(0) = u$,
$\Im(\G_{u,v}(s)) > 0$ for all $s > 0$, $\Im(\G_{u,v}(s)) < 0$ for all $s < 0$, and
$\lim_{s \to \pm \infty} |\G_{u,v}(s)| = \infty$. Moreover the contours are chosen not to intersect.
Then, recalling that $\Delta_n(y) = \det[ \psi_{n-l} (y_m) ]_{l,m=1}^n$ for all $y \in \R^n$,
\begin{eqnarray*}
K_n ((r,u),(s,v))
& = & \frac1{(2\pi)^2} \frac{(n-s)!}{(n-r-1)!} \int_{\g_{u,v}} dw \int_{\G_{u,v}} dz
\; \frac{(z - u)^{n-r-1}}{(w - v)^{n-s+1}} \frac1{w-z} \times \\
& \times & \frac1{n! B_n} \int_{\R^n} dy \; \left( \prod_{k=1}^n e^{-V(y_k)} \right)
\Delta_n(y)^2 \prod_{i=1}^n \left( \frac{w - y_i}{z - y_i} \right),
\end{eqnarray*}
for all $r \in \{1,\ldots,n-2\}$, $s \in \{1,\ldots,n\}$, and $u,v \in \R$ with $u \neq v$.
%
%
Then, recalling that $B_n = \prod_{k=0}^{n-1} c_k^2$ (see proof of Proposition \ref{prUIECorKer}),
Fyodorov and Strahov, \cite{Fyod03}, gives
\begin{eqnarray*}
K_n ((r,u),(s,v))
& = & \frac1{(2\pi)^2} \frac{(n-s)!}{(n-r-1)!} \int_{\g_{u,v}} dw \int_{\G_{u,v}} dz
\; \frac{(z - u)^{n-r-1}}{(w - v)^{n-s+1}} \frac1{w-z} \times \\
& \times & c_{n-1}^{-2} \int_{-\infty}^\infty dx \; e^{-V(x)} \;
\frac{\psi_{n-1}(x) \psi_n(w) - \psi_n(x) \psi_{n-1}(w)}{z-x},
\end{eqnarray*}
for all $r \in \{1,\ldots,n-2\}$, $s \in \{1,\ldots,n\}$, and $u,v \in \R$ with $u \neq v$.

\subsection{The classical ensembles}
\label{secTCE}

We end this paper by showing that the expression for the correlation kernel obtained
in Proposition \ref{prUIECorKer} in the special case of classical UIEs, agrees with the
expression obtained by Johansson and Nordenstam, \cite{Joh06a}, for the GUE
(see equation (\ref{eqJnGUE})). By classical UIEs we mean those that satisfy the
{\em generalised Rodrigues formula}:
\begin{hyp}
\label{hypUIEgrf}
There exists a function $\a : \R \to \R \setminus \{0\}$ for which
$\lim_{x \to \pm \infty} x^j \a(x)^k e^{-V(x)} = 0$ for all $j,k \in \N$. Also
there exists and a sequence, $\{a_{j,k}\}_{j,k \in \N} \subset \R \setminus \{0\}$,
for which
$$
\psi_j^{(k)}(x) = a_{j,k} \; \a(x)^{-k} e^{V(x)} \;
\frac{d^{j-k}}{dx^{j-k}} \left( \a(x)^j e^{-V(x)} \right),
$$
for all $x \in \R$ and $j,k \in \N$ with $j \ge k$.
\end{hyp}

First note an alternative expression for the correlation kernel in Proposition \ref{prUIECorKer}
can be obtained using equation $(3.1.12)$ of Szeg\"{o}, \cite{Sze39}. This gives
$$
1_{s > r} \frac{(v-u)^{s-r-1}}{(s-r-1)!}
= \frac{\partial^{n-s}}{\partial v^{n-s}}
\int_{-\infty}^\infty dx \; \frac{(x-u)^{n-r-1}}{(n-r-1)!}
\left( \sum_{j=0}^{n-1} c_j^{-2} \psi_j(x) \psi_j(v) e^{-V(x)} \right),
$$
for all $r \in \{1,\ldots,n-1\}$, $s \in \{1,\ldots,n\}$ and $u,v \in \R$.
Proposition \ref{prUIECorKer} then implies that
$$
K_n((r,u),(s,v))
= \sum_{j=n-s}^{n-1} c_j^{-2} \psi_j^{(n-s)} (v)
\int_u^\infty dx \frac{(x-u)^{n-r-1}}{(n-r-1)!} \psi_j (x) e^{-V(x)}
- \frac{(v-u)^{s-r-1}}{(s-r-1)!} 1_{v>u} 1_{s>r},
$$
for all $r \in \{1,\ldots,n-1\}$, $s \in \{1,\ldots,n\}$ and $u,v \in \R$.
%
%
Applying the Rodrigues formula inside the integral, and integrating by parts gives
\begin{eqnarray*}
\lefteqn{K_n((r,u),(s,v))
\; = \; (-1)^{n-r} \sum_{j=-\infty}^{n-1} \left( c_j^{-2} \frac{a_{j,0}}{a_{j,n-r}} \right)
\psi_j^{(n-r)}(u) \psi_j^{(n-s)} (v) \a(u)^{n-r} e^{-V(u)}} \\
& + & 1_{s>r} \left( \sum_{j=n-s}^{n-r-1} (-1)^j c_j^{-2} a_{j,0} \psi_j^{(n-s)} (v)
\int_u^\infty dx \frac{(x-u)^{n-r-1-j}}{(n-r-1-j)!} \a(x)^j e^{-V(x)}
- \frac{(v-u)^{s-r-1}}{(s-r-1)!} 1_{v>u} \right),
\end{eqnarray*}
for all $r \in \{1,\ldots,n-1\}$, $s \in \{1,\ldots,n\}$ and $u,v \in \R$.

In the special cases of the GUE, Laguerre and Jacobi ensembles, the above
correlation kernel agrees with that given in equation (4.15) of Forrester and Nagao,
\cite{Forr11}. In particular, in the GUE case, recall that (see, for example,
Anderson, Guionnet and Zeitouni, \cite{And10}) for all $j \in \N$ and $x \in \R$,
$\psi_j = H_j$ (the monic Hermite polynomial of degree $j$),
$c_j = \sqrt{\sqrt{2\pi} \; j!}$, $\a(x) = 1$, and $a_{j,k} = (-1)^{j-k} \frac{j!}{(j-k)!}$
for all $k \le j$. Also $H_{j+1}'(x) = (j + 1) H_j(x)$ for all $j \in \N$ and $x \in \R$.
%
%
Therefore
\begin{eqnarray*}
\lefteqn{K_n((r,u),(s,v))
\; = \; \sum_{j=-\infty}^{-1} \frac1{\sqrt{2\pi} (j+s)!} H_{j+r}(u) H_{j+s}(v) e^{-\frac12 u^2}} \\
& + & 1_{s>r} \left( \sum_{j=-s}^{-r-1} \frac1{\sqrt{2\pi} (j+s)!} H_{j+s}(v)
\int_u^\infty dx \frac{(x-u)^{-r-1-j}}{(-r-1-j)!} e^{- \frac12 x^2}
- \frac{(v-u)^{s-r-1}}{(s-r-1)!} 1_{v>u} \right),
\end{eqnarray*}
for all $r \in \{1,\ldots,n-1\}$, $s \in \{1,\ldots,n\}$ and $u,v \in \R$.
%
%
This agrees with the kernel given in equation (\ref{eqJnGUE}) (see remark \ref{remCorKer}),
and with equation (4.15) of Forrester and Nagao, \cite{Forr11}. Similarly for Jacobi and
Laguerre ensembles.

\vspace{0.5cm}

\textbf{Acknowledgements:} This research was carried out in University College Cork,
the Institut Mittag-Leffler, Universit\'{e} Paris VI, and the Royal Institute
of Technology (KTH). This research was partially supported by the G\"{o}ran Gustafsson
Foundation (KTH/UU). Special thanks to Benoit Collins, Neil O' Connell and Kurt Johansson,
and also to the anonymous referee and Associate Editor for helpful comments and suggestions
which have led to an improved version of the paper.

\end{document}